\documentclass[11pt,a4paper]{article}
\usepackage{amsmath,amsthm, amssymb, latexsym}
\usepackage{amsfonts}
\usepackage[T1]{fontenc}
\usepackage{lmodern}
\usepackage{graphicx}
\usepackage{amsfonts}
\usepackage{multicol}
\usepackage{caption}
\usepackage{color}
\usepackage{graphicx}
\usepackage{amsmath}
\usepackage{mathrsfs}
\usepackage{multicol}
\usepackage{float}
\usepackage{bbold}
\usepackage{color}
\usepackage{subfigure}
\usepackage{url}
\usepackage{tabularx}
\usepackage{xcolor}
\usepackage{grffile}
\usepackage{cancel}
\usepackage[english]{babel}
\usepackage{fullpage}
\usepackage{enumitem}
\usepackage[hidelinks]{hyperref}
\usepackage{array}
\usepackage{tabularx, caption}
\usepackage{cellspace}
\usepackage{todonotes}	
 \normalbaselines
\newtheorem{Theorem}{Theorem}[section]
\newtheorem{Definition}[Theorem]{Definition}
\newtheorem{Proposition}[Theorem]{Proposition}
\newtheorem{Lemma}[Theorem]{Lemma}
\newtheorem{Corollary}[Theorem]{Corollary}

\newtheorem{Remark}[Theorem]{Remark}

\newcommand{\N}{\mathbb{N}}
\newcommand{\Z}{\mathbb{Z}}
\newcommand{\R}{\mathbb{R}}
\newcommand{\C}{\mathbb{C}}

\newcommand{\norm}[1]{\left\lVert #1 \right\rVert}	
\newcommand{\normabs}[1]{\left| #1 \right|}	
\newcommand{\abs}[1]{\left| #1 \right|}

\newcommand\restr[2]{{% we make the whole thing an ordinary symbol
  \left.\kern-\nulldelimiterspace % automatically resize the bar with \right
  #1 % the function
  \vphantom{\big|} % pretend it's a little taller at normal size
  \right|_{#2} % this is the delimiter
  }}
\allowdisplaybreaks[4]
\newcommand{\mylabel}[2]{#2\def\@currentlabel{#2}\label{#1}}

\begin{document}

\title{Continuous-time locally stationary time series models\\
}
\author{Annemarie Bitter, Robert Stelzer and Bennet Ströh\footnote{Ulm University, Institute of Mathematical Finance, Helmholtzstra\ss e 18, 89069 Ulm, Germany. Emails: bitter.annemarie@gmail.com, robert.stelzer@uni-ulm.de, bennet.stroeh@uni-ulm.de. }}

%\date{}
\maketitle

\textwidth=160mm \textheight=225mm \parindent=8mm \frenchspacing
\vspace{3mm}

\begin{abstract}
We adapt the classical definition of locally stationary processes in discrete-time (see e.g. \cite{D2012}) to the continuous-time setting and obtain equivalent representations in the time and frequency domain. From this, a unique time-varying spectral density is derived using the Wigner-Ville spectrum. As an example, we investigate time-varying L\'evy-driven state space processes, including the class of time-varying L\'evy-driven CARMA processes. First, the connection between these two classes of processes is examined. Considering a sequence of time-varying L\'evy-driven state space processes, we then give sufficient conditions on the coefficient functions that ensure local stationarity with respect to the given definition.
\end{abstract}


\section{Introduction}
\label{sec1}
To model non-stationary time series that behave locally in a stationary manner, Dahlhaus and others developed, starting with the seminal paper \cite{D1996}, a comprehensive theory and powerful estimation procedures, using a parameterized sequence of processes for the definition of local stationarity (see e.g. \cite{D1997, D2000, DP2009}, or \cite{D2012} for an overview). Noticeable examples include, for instance, ARMA processes with continuous coefficient functions (see \cite{D1996}).
More recently, also non-parametric approaches that allow for linear and non-linear locally stationary models were introduced and investigated in \cite{BDW2020, DRW2019, DSR2006, V2012, VD2015}.\\
Despite this success, the above approaches have just been carried out for models defined on $\Z$, i.e. in a discrete-time framework. Surprisingly, there is so far no theory for locally stationary models defined on $\R$, i.e. in a continuous-time framework, available.\\
In this paper, we tackle this issue and define local stationarity for continuous-time models following the original definition from Dahlhaus \cite{D1996}. We establish such a definition in the frequency and time domain and, as we consistently use $L^2$-integration theory (see e.g. \cite{A2009} for an introduction), we readily obtain that both definitions are equivalent. Based on the definition in the frequency domain, we define a time-varying spectral density and show that it can be uniquely determined by a sequence of locally stationary processes, using the Wigner-Ville spectrum (see also \cite{D1996}). This uniqueness is a powerful property, as it is known to pave the way for a likelihood approximation (comparable to the Whittle likelihood for stationary processes), leading to powerful estimation methods (see \cite{D2000}).\\
As an example, we consider time-varying L\'evy-driven state space processes, which include the continuous-time analog of time-varying ARMA (time-varying CARMA) processes. L\'evy-driven CARMA processes are known to provide a flexible yet analytically tractable class of processes that have been applied to model a variety of phenomena from different areas \cite{BKMV2014, LM2004, MS2007}.\\
In the time-invariant setting, it is known from \cite{SS2012} that the class of CARMA processes is equivalent to the class of L\'evy-driven state space processes. While it is easy to see that also every time-varying CARMA process is a time-varying L\'evy-driven state space process, we show that the inverse inclusion fails to hold, at least for non-continuous coefficient functions. This motivates to look at the class of time-varying L\'evy-driven state space models.\\
More precisely, the paper is structured as follows. In Section \ref{sec2}, we first review the definition of local stationarity in the discrete-time framework. Then in Section \ref{sec2-2}, we summarize basic facts about L\'evy processes and orthogonal random measures, including integration with respect to them.\\
The novel definition of local stationarity for continuous-time models both in the frequency and time domain is given in Section \ref{sec3}. Moreover, we investigate asymptotic distributional properties of such models and show that the autocovariance function evaluated at distinct points tends to zero.\\ 
In Section \ref{sec4}, we investigate time-varying state space processes in the context of local stationarity. We start with a simple example in Section \ref{sec4-1}, where we consider a sequence of time-varying CAR(1) processes and give sufficient conditions on the coefficient function such that the sequence is locally stationary according to the given definition. Section \ref{sec4-2} and \ref{sec4-3} are dedicated to general time-varying state space processes. First, in Section \ref{sec4-2}, the connection between the class of time-varying CARMA processes and time-varying state space processes is examined. Then, we give sufficient conditions for a sequence of time-varying state space processes to be locally stationary.\\
%Section~\ref{sec4-3} refers to
%the general setting of time-varying state space models.
%Section~\ref{sec_4___LocStatExamples} considers locally stationary state space processes, including the class of CARMA processes. Firstly the connection between these two classes of processes is examined and secondly suitable conditions on the coefficient functions are given, such that they fulfill the definition of locally stationarity. 
Finally, in Section \ref{sec5} we investigate the time-varying spectral density and the Wigner-Ville spectrum of locally stationary processes.

\section{Preliminaries} 
\label{sec2}
Throughout this paper, we denote the set of positive integers by $\N$, non-negative real numbers by $\R^+_0$, the set of $m\times n$ matrices over a ring $R$ by $M_{m\times n}(R)$ and $\mathbf{1}_n$ stands for the $n\times n$ identity matrix. Given a complex number $z$, we denote the complex conjugate of $z$ by $\overline{z}$. For square matrices $A,B\in M_{n\times n}(R)$, $[A,B]=AB-BA$ denotes the commutator of $A$ and $B$, $Rank(A)$ the rank of $A$ and $\sigma(B)$ the spectrum of $B$. We shortly write the transpose of a matrix $A \in M_{m\times n}(\R)$ as $A'$ and the adjoint of a matrix $B \in M_{m\times n}(\C)$ as $B^*$. Norms of matrices and vectors are denoted by $\norm{\cdot}$. If the norm is not further specified, we take the Euclidean norm or its induced operator norm, respectively. For a complex number $z\in\C$, the real part of $z$ is denoted by $\mathfrak{Re}(z)$.
The Borel $\sigma$-algebras are denoted by $\mathcal{B}(\cdot)$ and $\lambda$ stands for the Lebesgue measure, at least in the context of measures. In the following, we will assume all stochastic processes and random variables to be defined on a common complete probability space $(\Omega,\mathcal{F},P)$ equipped with an appropriate filtration if necessary. 
We simply write $L^p$ to denote the space $L^p(\Omega,\mathcal{F},P)$ and $L^p(X)$ to denote the space $L^p(X,\mathcal{B}(X),\lambda)$ for some set $X\subset\R$ with corresponding norms $\norm{\cdot}_{L^p}$. The ring of continuous functions in $t$ from $\R$ to $\R$ is denoted by $\mathcal{R}[t]$.

\subsection{Locally stationary time series in discrete time} 
\label{sec2-1}

We follow the concept of local stationarity as established in \cite{D2012} for discrete-time locally stationary time series models. There, the authors considered a parametric representation of a sequence of non-stationary time-varying processes either in the time or frequency domain, which has to satisfy certain regularity conditions.\\
In the following we briefly review the mathematical details of the aforementioned concepts as well as the most important results.
To this end, we define the total variation of a function $g$ on $[0,1]$, denoted by $V(g)$, as
\begin{align*}
V(g):=\sup\left\{ \sum_{k=1}^m |g(x_k)-g(x_{k-1})|, 0\leq x_0 < \ldots < x_m \leq 1, m\in\N \right\}
\end{align*}
and for $\kappa>0$ we define
\begin{align*}
\ell_\kappa(j):=\begin{cases} 1,& |j|\leq1, \\
|j|\log^{1+\kappa}|j|, 	& |j|>1
\end{cases}
\end{align*}
for all $j\in\Z$. For further details on the following two definitions we refer to \cite{D2012}.

\begin{Definition}
\label{def_locstat_discrete1}
Let $\{X_{t,T}, t=1,\ldots,T\}_{T\in\N}$ be a sequence of stochastic processes. Then, $X_{t,T}$ is called locally stationary in the time domain, if there exists a representation 
\begin{align*}
X_{t,T}=\sum_{j=-\infty}^{\infty} a_{t,T,j}\varepsilon_{t-j}, \qquad T\in\N, \ t=1,...,T
\end{align*}
where
\begin{enumerate}[label={(\alph*)}]
\item $\{\varepsilon_t,t\in\Z\}$ is an i.i.d. (independent identically distributed) sequence with $E[\varepsilon_t]=0$ and $Var(\varepsilon_t)=1$,
\item for all $j\in\Z$ it holds
\begin{align*}
\sup_{\substack{t=1,...,T\\T\in\N}}|a_{t,T,j}| \leq \frac{K}{\ell_\kappa(j)},
\end{align*}
where $\kappa,K>0$ are constants and 
\item there exist functions $a_j(\cdot):(0,1]\rightarrow\R$, $j\in\Z$, satisfying
\begin{align}
\sup_{u\in(0,1]}|a_j(u)|\leq \frac{K}{\ell_\kappa(j)}, \qquad
\sup_{j\in\Z}\sum_{t=1}^T|a_{t,T,j}-a_j(\tfrac{t}{T})| \leq K \label{def_locstat_discrete1_a_j}\text{ and}\qquad
V(a_j(\cdot)) \leq \frac{K}{\ell_\kappa(j)}
\end{align}
for some constant $K$.
\end{enumerate} 
\end{Definition}

\begin{Definition}
\label{def_locstat_discrete2}
Let $\{X_{t,T}: t=1,\ldots,T\}_{T\in\N}$ be a sequence of stochastic processes. Then, $X_{t,T}$ is called locally stationary in the frequency domain with transfer functions $A_{t,T}^0:[-\pi,\pi]\rightarrow\C$, $T\in\N$, $t=1,\ldots,T$, if it has the representation 
\begin{align*}
X_{t,T}=\int_{-\pi}^{\pi} e^{i\lambda t} A_{t,T}^0(\lambda) \xi(d\lambda) \qquad \text{for all } T\in\N,\ t=1,\ldots,T,
\end{align*}
(with the integrals existing in $L^2$) where
\begin{enumerate}[label={(\alph*)}]
\item $\xi(\lambda)$ is a stochastic process on $[-\pi,\pi]$ with mean zero and orthogonal increments, %$(\xi(\lambda))^*=\xi(-\lambda)$ and \linebreak
			%		$cum\{d\xi(\lambda_1) , ... , d\xi(\lambda_k)\} = \eta \left( \sum_{j=1}^k \lambda_j  \right) h_k(\lambda_1 , ... , \lambda_{k-1}) d\lambda_1 ... d\lambda_k$, 
			%		where $cum\{...\}$ denotes the k-th order cumulant, $h_1=0$, $h_2(\lambda)=1$, $|h_k(\lambda_1,...,\lambda_{k-1})| \leq const_k$ for all $k$ and 
			%		$\eta(\lambda)=\sum_{j=-\infty}^\infty \delta(\lambda+2\pi j)$ is the period $2\pi$ extension of the Dirac delta function. 
					
\item there exists a constant $K$ and a function $A:[0,1]\times[-\pi,\pi]\rightarrow\C$, which is continuous in the first component satisfying $\overline{A(u,\lambda)}=A(u,-\lambda)$ and
\begin{align}\label{equ_locstat_discrete2_A0_A}
\sup_{\substack{t=1,...,T, \\ \lambda\in[-\pi,\pi]}}\big|A_{t,T}^0(\lambda)-A(\tfrac{t}{T},\lambda)\big| \leq \frac{K}{T}, \qquad T\in\N.
\end{align}
\end{enumerate}
\end{Definition}

\begin{Remark}
\begin{enumerate}[label={(\alph*)}]
\item Due to the smoothness conditions on the coefficient functions $a_j(u)$ and the transfer function $A(u,\lambda)$, the sequence $X_{t,T}$ shows a locally stationary behavior (see e.g. \cite[Definition 2.1]{D1997}). 
%The smoothness of the coefficient functions $a_j(u)$ respectively the transfer function $A(u,\lambda)$ ensures, that the behavior of the sequence $X_{t,T}$ locally looks stationary (e.g. \cite[Definition 2.1]{D1997}). %, since by \cite{brockwell_time_1996} p.~142 the stochastic integral establishes a continuity property in mean squared w.r.t. the integrated function.
	%\textcolor[rgb]{1,0,0}{That is, if $u_n\rightarrow u\in\R$ ($n\rightarrow\infty$) such that $\|A(u_n,\cdot)-A(u,\cdot)\|_{L^2([-\pi,\pi],\C)}\rightarrow0$, then for 
	%$X_t(u):=\int_{(-\pi,\pi]} e^{i\lambda t} A(u,\lambda)\xi(d\lambda)$ it holds, that
	%\[  Cov\left(X_t(u_n) , X_{t+h}(u_n)\right) \rightarrow Cov\left(X_t(u) , X_{t+h}(u)\right) 
	%\quad \text{as } n\rightarrow\infty. \]
	%}
\item For a comprehensive introduction to orthogonal increment processes, orthogonal random measures and the related $L^2$-integration theory we refer to \cite{BD1996}.
%A concise account on orthogonal increment processes/orthogonal random measures and the related $L^2$-integration theory can be found in.
\item We note that the given definitions of local stationarity in the time and frequency domain are not equivalent.\\ %If not otherwise stated, we just speak of ``local stationarity'' in both cases.\\
However, using the spectral representation of the noise $\varepsilon_t = \int_{(-\pi,\pi]} \tfrac{1}{\sqrt{2\pi}} e^{i\lambda t} \xi(d\lambda)$ (see \cite{BD1996}), the Fourier transform allows for the following connections (see \cite[Remark 2.2]{D2000}) between the two concepts. It holds
\begin{align*}
A_{t,T}^0(\lambda) 	&= \frac{1}{\sqrt{2\pi}} \sum_{j=-\infty}^{\infty} a_{t,T,j} e^{-i\lambda j}, \qquad & \qquad
A(u,\lambda) 		&= \frac{1}{\sqrt{2\pi}} \sum_{j=-\infty}^{\infty} a_j(u) e^{-i\lambda j}, \\
a_{t,T,j}			&= \frac{1}{\sqrt{2\pi}} \int_{-\pi}^{\pi} A_{t,T}^0(\lambda) e^{i\lambda j} d\lambda\text{ and} \qquad & \qquad
a_j(u)			&= \frac{1}{\sqrt{2\pi}} \int_{-\pi}^{\pi} A(u,\lambda) e^{i\lambda j} d\lambda,
\end{align*}
since $A_{t,T}^0(\cdot)\in L^2([-\pi,\pi],\C)$ and $a_{t,T,j}\in\ell^2$. \\
%\cite[Remark 2.2]{DP2009} states that, if (\ref{def_locstat_discrete1_a_j}) is replaced by a stronger condition the assumptions of Definition \ref{def_locstat_discrete2} hold.\\
Necessary conditions for Definition \ref{def_locstat_discrete1} and \ref{def_locstat_discrete2} to be equivalent can be found in \cite[Remark 2.2]{DP2009}. In particular, this includes additional smoothness assumptions on $A(u,\lambda)$ and a stronger version of the second condition in (\ref{def_locstat_discrete1_a_j}).
%
%the assumptions of Definition \ref{def_locstat_discrete1} are fulfilled. 
%The relation between both definitions is given by the spectral representation of the noise $\varepsilon_t = \int_{(-\pi,\pi]} \tfrac{1}{\sqrt{2\pi}} e^{i\lambda t} \xi(d\lambda)$ (see \cite{BD1996}) and the Fourier transforms, see also \cite[Remark 2.2]{D2000}:
%\begin{align*}
%A_{t,T}^0(\lambda) 	&= \frac{1}{\sqrt{2\pi}} \sum_{j=-\infty}^{\infty} a_{t,T,j} e^{-i\lambda j} \qquad & \qquad
%A(u,\lambda) 				&= \frac{1}{\sqrt{2\pi}} \sum_{j=-\infty}^{\infty} a_j(u) e^{-i\lambda j} \\
%a_{t,T,j}						&= \frac{1}{\sqrt{2\pi}} \int_{-\pi}^{\pi} A_{t,T}^0(\lambda) e^{i\lambda j} d\lambda 
%\qquad & \qquad
%a_j(u)							&= \frac{1}{\sqrt{2\pi}} \int_{-\pi}^{\pi} A(u,\lambda) e^{i\lambda j} d\lambda,
%\end{align*}
%since $A_{t,T}^0(\cdot)\in L^2([-\pi,\pi],\C)$ and $a_{t,T,j}\in\ell^2$. 
\end{enumerate}
\end{Remark}

%Since it is the aim of this paper to establish a definition of local stationarity in the continuous-time setting similar to Definition \ref{def_locstat_discrete1} and \ref{def_locstat_discrete2} 
The following two propositions give further insight into Definition \ref{def_locstat_discrete2} and the notion of local stationarity.% As usual we define the $L^2$ norm of a function $f:(-\pi,\pi]\rightarrow \C$  as $\norm{f(\cdot)}_{L^2((-\pi,\pi],\C)}=\int_{(-\pi,\pi]} |f(x)|^2 dx$.

\begin{Proposition}\label{proposition:L2convergencediscretetime}
Let $X_{t,T}$ be a locally stationary process in the frequency domain and $\{T_n\}_{n\in\N}\subset\N$ an increasing sequence. If $sT_n\in\{1,...,T_n\}$ for some fixed $s\in[0,1]$ and all $n>n_0$, $n_0\in\N$, then it holds
\begin{align*}%\label{equation:AsTnconverges}
A_{sT_n,T_n}^0(\cdot) \xrightarrow[n\rightarrow\infty]{L^2} A(s,\cdot).
\end{align*}
\label{prop_1_1_L2norm}
\end{Proposition} 
\begin{proof}
Follows directly from (\ref{equ_locstat_discrete2_A0_A}).
\end{proof}

\noindent For instance, the choice $T_n=2^n$ and $s=k/2^{n_0}$ for some $n_0\in\N$ and $k\in\{1,...,T_{n_0}\}$ suits the conditions of Proposition \ref{proposition:L2convergencediscretetime}.% such that (\ref{equation:AsTnconverges}) holds.

\begin{Proposition} 
Let $X_{t,T}$ be a locally stationary process in the frequency domain with associated orthogonal increment process $\{\xi(\lambda),\lambda\in[-\pi,\pi]\}$ corresponding to an i.i.d. noise (i.e. $\varepsilon_t = \int_{(-\pi,\pi]} \tfrac{1}{\sqrt{2\pi}} e^{i\lambda t} \xi(d\lambda)$ defines an i.i.d. noise) and $\{T_n\}_{n\in\N}\subset\N$ an increasing sequence. If $sT_n\in\{1,...,T_n\}$ for some fixed $s\in[0,1]$ and all $n>n_0$, $n_0\in\N$, then it holds
\begin{align*}
X_{sT_n,T_n}\overset{d}{\underset{n\rightarrow\infty}{\longrightarrow}}\int_{-\pi}^{\pi} A(s,\lambda) \xi(d\lambda).
\end{align*}
\label{prop_1_1_ConvergenceInDistribution}
\end{Proposition} 

\begin{proof}
First observe that every time series of the form $\int_{-\pi}^{\pi} e^{i\lambda t} A(\lambda) \xi(d\lambda)$, $t\in\Z$, where $\xi$ is an orthogonal increment process coming from an i.i.d. noise, is strictly stationary. Thus,
\begin{align*}
\int_{-\pi}^{\pi} e^{i\lambda t_1} A_{sT_n,T_n}^0(\lambda) \xi(d\lambda) 
\overset{d}{=} \int_{-\pi}^{\pi} e^{i\lambda t_0} A_{sT_n,T_n}^0(\lambda) \xi(d\lambda)
\end{align*}
for all $t_0,t_1\in\Z$. In particular, for $t_1=sT_n$, where $s\in[0,1]$ such that $sT_n\in\{1,\ldots,T\}$, and $t_0=0$ we obtain
\begin{align*}
X_{sT_n,T_n}	=\int_{-\pi}^{\pi} e^{i\lambda sT_n} A_{sT_n,T_n}^0(\lambda) \xi(d\lambda) \overset{d}{=} \int_{-\pi}^{\pi} A_{sT_n,T_n}^0(\lambda) \xi(d\lambda).
\end{align*}
The remainder follows from Proposition \ref{prop_1_1_L2norm} and the continuity of the stochastic integral in mean square and thus in distribution with respect to the integrand.
\end{proof}

\begin{Remark}
A noticeable class of processes that are locally stationary in the time and frequency domain are time-varying AR(p) processes with continuous coefficient functions. For the mathematical details of this result we refer to \cite[p. 147]{D1996}.\\
Among the variety of different concepts for local stationarity in the literature, we mention the results from \cite{DRW2019} and \cite{V2012}. In \cite{V2012}, the author considers a triangular array $X_{t,T}$, $T\in\N$, $t=1,\ldots,T$ to be locally stationary, if for each $u\in[0,1]$ there exists a strictly stationary process $\{X_t(u),t=1,\ldots,T\}$ such that almost surely
\begin{align*}
\abs{X_{t,T}-X_t(u)}\leq (\abs{\tfrac{t}{T}-u}+\tfrac{1}{T}) U_{t,T}(u),
\end{align*}
where $U_{t,T}(u)$ are positive random variables satisfying $E[(U_{t,T}(u))^\rho]<\infty$ for some $\rho>0$ uniformly in $u,t$ and $T$.
Time-varying AR(p) processes with continuous coefficient functions can also be embedded in this framework using similar arguments as in \cite{DSR2006}.\\
More recently, the authors in \cite{DRW2019} developed a general theory for locally stationary processes based on stationary approximations. Similarly to \cite{V2012}, it is assumed that there exists a strictly stationary process $\{X_t(u),t=1,\ldots,T\}$ such that for some $q,C>0$
\begin{align}\label{equation:DRWinequalities}
\norm{X_t(u)-X_t(v)}_{L^q}\leq C \abs{u-v} \quad \text{and}\qquad \norm{X_{t,T}-X_t(\tfrac{t}{T})}_{L^q}\leq \tfrac{C}{T},
\end{align}
uniformly in $t=1,\ldots,T$ and $u,v\in[0,1]$. Based on these approximations the authors established asymptotic results as a law of large numbers and a central limit theorem, which, in turn, are used to derive asymptotic results for a maximum likelihood estimator (see \cite[Section 5]{DRW2019}). Again, time-varying AR(p) processes with continuous coefficient functions can be embedded in this framework. Recently, this work has been extended to models with infinite memory in \cite{BDW2020}.\\
In view of the statistical results obtained from the approximations (\ref{equation:DRWinequalities}), a possible characterization of local stationarity in terms of similar approximations for continuous-time models will be the topic of future work.
\end{Remark}

\subsection{L\'evy processes and orthogonal random measures}
\label{sec2-2}
In this section we lay the foundation for the definition of continuous-time locally stationary processes and briefly review L\'evy processes, orthogonal random measures and cover basic results including stochastic integration with respect to L\'evy processes and orthogonal random measures. For further insight we refer to \cite{A2009} and \cite{S2013}.

\begin{Definition}
A real-valued stochastic process $L=\{L(t),t\in \R_0^+\}$ is called L\'evy process if
\begin{enumerate}[label={(\alph*)}]
\item $L(0)=0$ almost surely,
\item for any $n\in\N$ and $t_0<t_1<t_2<\dots<t_n$, the random variables $(L(t_0),L(t_1)-L(t_0),\dots,L(t_n)-L(t_{n-1}))$ are independent,
\item for all $s,t \geq 0$, the distribution of $L(s+t)-L(s)$ does not depend on $s$ and
\item $L$ is stochastically continuous.
\end{enumerate}
\end{Definition}

\begin{Theorem}\label{theorem:levykhin}
Let $L=\{L(t),t\geq0\}$ be a real-valued L\'evy process. Then, $L(1)$ is an infinitely divisible real-valued random variable with characteristic triplet $(\gamma,\Sigma,\nu)$, where $\gamma \in \R$, $\Sigma>0$ and $\nu$ is a L\'evy measure on $\R$, i.e. $\nu(0)=0$ and $\int_{\R}(1\wedge\normabs{x}^2)\nu(dx)<\infty$. The characteristic function of $L(t)$ is given by
\begin{align}\label{eq:levykhintchine}
\begin{aligned}
\varphi_{L(t)}(z)&=E[e^{izL(t)}]=\exp(t \Psi_{L}(z)),\\
\Psi_L(z)&=\left( i\gamma z -\frac{\Sigma z^2}{2}+\int_{\R } \left(e^{i z x}-1-i z x \mathbb{1}_{Z}(x)	\right)\nu(dx)\right),
\end{aligned}
\end{align}
where $z \in \R$ and $Z=\{x \in \R, \normabs{x}\leq 1\}$.
\end{Theorem}

 In the remainder we work with two-sided L\'evy process, i.e. $L(t)=L_1(t)\mathbb{1}_{\{t\geq0\}} - L_2(-t)\mathbb{1}_{\{t<0\}}$, where $L_1$ and $L_2$ are independent copies of a one-sided L\'evy process. Throughout this paper, it will be assumed that 
\begin{align}\label{ass_L}
		E[L(1)]=\gamma + \int_{|x|> 1} x \nu(dx)=0 
		\text{\quad and \quad} 
		E[L(1)^2]=\Sigma + \int_{x\in\R} x^2 \nu(dx)<\infty.
\end{align}
Thus, the above assumptions on the L\'evy process imply that $\int_\R x^2 \nu(dx)<\infty$. %, since by \cite[Theorem 25.3]{S2013} a L\'evy process has finite second moments if and only if  $\int_{|x|>1} x^2 \nu(dx)<\infty$. 
Occasionally, we will denote $\Sigma_L:=Var(L(1))=\Sigma+\int_{x\in\R} x^2 \nu(dx)$. \\
If the L\'evy process satisfies (\ref{ass_L}) and $f:\R\times\R\rightarrow \R$ is a $\mathcal{B}(\R\times \R) - \mathcal{B}(\R)$-measurable function satisfying $f(t,\cdot)\in L^2(\R)$, then the integral $X(t)=\int_{\R} f(t,s) L(ds)$, $t\in\R$ exists in $L^2$ (see e.g. \cite{MS2007}).

\begin{Definition} [{\cite[Definition 2.3.5]{K2002}}]
A family $\{\xi(\Delta)\}_{\Delta\in\mathcal{B}(\R)}$ of $\C$-valued random variables is called an orthogonal random measure (ORM) if
\begin{enumerate}[label={(\alph*)}]
\item $\xi(\Delta)\in L^2(\mathcal{B}(\R),\C)$ for all bounded $\Delta\in\mathcal{B}(\R)$,
\item $\xi(\emptyset)=0$,
\item $\xi(\Delta_1\cup\Delta_2)=\xi(\Delta_1)+\xi(\Delta_2)$ a.s. whenever $\Delta_1\cap\Delta_2=\emptyset$ and
\item $F:\mathcal{B}(\R)\rightarrow \C$ such that $F(\Delta)=E[\xi(\Delta)\overline{\xi(\Delta)}]$ defines a $\sigma$-additive positive definite measure and it holds that $E[\xi(\Delta_1)\overline{\xi(\Delta_2)}]=F(\Delta_1\cap\Delta_2)$ for all $\Delta_1,\Delta_2\in\mathcal{B}(\R)$.% (in particular: $E[\xi(\Delta_1)\overline{\xi(\Delta_2)}]=0$ for $\Delta_1\cap\Delta_2=\emptyset$).
\end{enumerate}
\end{Definition} 

\noindent$F$ is referred to as the spectral measure of $\xi$.

\begin{Theorem} [{\cite[Theorem 3.5]{MS2007} }]
\label{thm_Levy_ORM}
Let $L$ be a two-sided L\'evy process satisfying (\ref{ass_L}). Then, there exists an ORM $\Phi_L$ with spectral measure $F_L$, such that 
\begin{enumerate}[label={(\alph*)}]		
\item $E[\Phi_L(\Delta)]=0$ for any bounded $\Delta\in\mathcal{B}(\R)$,
\item $F_L(dt)=\frac{\Sigma_L}{2\pi}$dt and
\item $\Phi_L$ is uniquely determined by $\Phi_L([a,b)):=\int_{-\infty}^\infty \frac{e^{-i\mu a}-e^{-i\mu b}}{2\pi i\mu} L(d\mu)$.
\end{enumerate}
\end{Theorem}

\noindent In the proof of the above theorem the standard theory of Fourier transforms on $L^2(\R)$ (see e.g. \cite[Chapter 2]{C1989} for an introduction) is used to show
\begin{align}\label{eq_link_L_ORM}
\int_{-\infty}^\infty \varphi(\mu) \Phi_L(d\mu)=\frac{1}{\sqrt{2\pi}}\int_{-\infty}^\infty \widehat{\varphi}(u) L(du)
\end{align}
for all complex functions $\varphi\in L^2(\R)$ and their (inverse) Fourier transforms $\widehat{\varphi}$, where
\begin{align*}
\widehat{\varphi}(u)=\frac{1}{\sqrt{2\pi}}\int_{-\infty}^\infty e^{-i\mu u} \varphi(\mu) d\mu
\quad \text{and} \quad
\varphi(\mu)=\frac{1}{\sqrt{2\pi}}\int_{-\infty}^\infty e^{i\mu u} \widehat{\varphi}(u) du.
\end{align*}
%The above definition of the integral refers to the Plancherel extension theorem of the Fourier transform on $L^1(\R,\C)\cap L^2(\R,\C)$ to $L^2(\R,\C)$ (since $L^1(\R,\C)\cap L^2(\R,\C)$ is dense in $L^2(\R,\C))$ see \cite{katznelson_introduction_2004} Theorem~3.1, p.~176, and \cite{malliavin_integration_1995} Theorem~2.4.9, p.~132. \\
%As main argument the Fourier (Plancherel) transform defined on $L^2$ is used.
We also recall that for two complex functions $f,g\in L^2(\R)$ and their Fourier transforms $\widehat{f}, \widehat{g}$, it follows that $\hat{f}, \hat{g}\in L^2(\R)$ and, due to \cite[p. 189]{R1991},
\begin{align*}
\int_{-\infty}^\infty f(\mu)\overline{g(\mu)} d\mu = \int_{-\infty}^\infty \widehat{f}(u)\overline{\widehat{g}(u)} du.
\end{align*}

\section{Locally stationary processes in continuous-time}
\label{sec3}
Analogously to Section \ref{sec2-1} one can define a (stationary) stochastic process $\{Y(t)\}_{t\in\R}$ via the representation as a linear process or the spectral representation, i.e.
\begin{align*}
Y(t)	= \int_{\R} g(t-u) L(du) \quad \text{or} \quad
Y(t) 	= \int_{\R} e^{i\mu t} A(\mu) \Phi_L(d\mu), \qquad t\in\R,
\end{align*}
where $g$ and $A$ are square integrable functions and $L$ is a two-sided L\'evy process with corresponding ORM $\Phi_L$. As we consistently use $L^2$-integrals to define the process both in the time and the frequency domain and the Fourier transform is an isometry on $L^2$ the two definitions are equivalent. Hence, from (\ref{eq_link_L_ORM}) it follows that the transfer function $A$ and the kernel $g$ satisfy % for fixed $t\in\R$
\begin{align*}
g(u) = \frac{1}{2\pi} \int_{-\infty}^\infty e^{i\mu u} A(\mu) d\mu \quad \text{and} \quad
A(\mu)= \int_{-\infty}^\infty e^{-i\mu v} g(v) dv.
\end{align*}
Now, we allow the kernel function and the transfer function be time dependent, leading to
\begin{align*}
Y(t) = \int_{\R} e^{i\mu t} A(t,\mu) \Phi_L(d\mu) = \int_{\R} g(t,t-u) L(du), \qquad t\in\R,
\end{align*}
where
\begin{align*}	
g(t,u)	= \frac{1}{2\pi} \int_{-\infty}^\infty e^{i\mu u} A(t,\mu) d\mu\quad \text{and} \quad
A(t,\mu)	= \int_{-\infty}^\infty e^{-i\mu u} g(t,u) du.
\end{align*}
As we are interested in real-valued processes we demand $g$ to be real-valued or equivalently that $\overline {A(\mu)}=A(-\mu)$ for all $\mu\in\mathbb{R}$.\\
To be able to define local stationarity analogously to Section \ref{sec2-1}, not only a time varying representation, but also a sequence of stochastic processes is needed. The intuitive idea is to take a limiting kernel $g$ and a sequence of kernels $g_N^0$ defining the processes in the time domain such that %Proceeding similarly to \cite{dahlhaus_likelihood_2000} and Proposition~\ref{prop_1_1_L2norm}, \textcolor[rgb]{1,0,0}{we begin with rescaling the coefficients in the kernel, which yields 
\begin{align*}
\norm{g_N^0(t,\cdot)-g(\tfrac{t}{N},\cdot)}_{L^2} \underset{N\rightarrow\infty}{\longrightarrow} 0.
\end{align*}
However, for the limiting (stationary) process we prefer to fix a time $t\in\R$ rather than dealing with fractions $\tfrac{t}{N}$. By replacing $t$ by $Nt$ this leads to the following definition.

\begin{Definition}\label{def_locstat_cont}
A sequence of stochastic processes $\{Y_N(t),t\in\R\}_{N\in\N}$ is said to be locally stationary in the time domain, if it can be represented as
\begin{align*}
Y_N(t) = \int_{\R} g_N^0(Nt,Nt-u) L(du),\quad \text{for all }t\in\R, N\in\N,
\end{align*}
where $L$ is a two-sided L\'evy process and the kernel functions $g_N^0:\R\times\R\rightarrow\R$ satisfy 
\begin{enumerate}[label={(\alph*)}]
\item $g_N^0(Nt,\cdot)\in L^2(\R)$ for all $t\in\R, N\in\N$ and
\item there exists a (local/limiting kernel) function $g:\R\times\R\rightarrow\R$ such that the mapping $\R\rightarrow L^2(\R)$, $t\mapsto g(t,\cdot)$ is continuous and 
\begin{align*}
g_N^0(Nt,\cdot) \underset{N\rightarrow\infty}{\overset{L^2}{\longrightarrow}} g(t,\cdot) \quad \text{for all } t\in\R.
\end{align*}
\end{enumerate}
\end{Definition}

\begin{Definition}\label{def_locstat_cont2}
A sequence of stochastic processes $\{Y_N(t),t\in\R\}_{N\in\N}$ is said to be locally stationary in the frequency domain, if it can be represented as
\begin{align}\label{equation:locallystationaryprocessfrequencydomain}
Y_N(t) = \int_{\R} e^{i\mu Nt} A_N^0(Nt,\mu) \Phi_L(d\mu),\quad \text{for all }t\in\R, N\in\N,
\end{align}
where $\Phi_L$ is the ORM of a two-sided L\'evy process $L$ and the transfer functions $A_N^0:\R\times\R\mapsto\C$ satisfy
\begin{enumerate}[label={(\alph*)}]
\item $A_N^0(Nt,\cdot)\in L^2$ for all $t\in\R, N\in\N$, 
\item $\overline{A_N^0(\cdot,\cdot)}=A_N^0(\cdot,-\cdot)$ and
\item there exists a (local/limiting transfer) function $A:\R\times\R\rightarrow\C$ with $\overline{A(\cdot,\cdot)}=A(\cdot,-\cdot)$ such that the mapping $\R \mapsto L^2(\R)$, $t\mapsto A(t,\cdot)$ is continuous and
\begin{align*}
A_N^0(Nt,\cdot) \underset{N\rightarrow\infty}{\overset{L^2}{\longrightarrow}} A(t,\cdot),\quad \text{for all } t\in\R.
\end{align*}
\end{enumerate}
\end{Definition}

In contrast to the discrete time case it is now irrelevant whether we use the definition in the time or the frequency domain. Therefore, we will just speak of ``locally stationary'' in both cases. 

\begin{Proposition} 
The Definitions \ref{def_locstat_cont} and \ref{def_locstat_cont2} are equivalent. Moreover, the relationship between the (limiting) transfer function and the (limiting) kernel, using their Fourier transforms, is given by
\begin{align*}
A_N^0(Nt,\mu) &= \int_{-\infty}^\infty e^{-i\mu u} g_N^0(Nt,u) du ,\qquad & \qquad A(t,\mu) &= \int_{-\infty}^\infty e^{-i\mu u} g(t,u) du,\\	g_N^0(Nt,u) &= \frac{1}{2\pi} \int_{-\infty}^\infty e^{i\mu u} A_N^0(Nt,\mu) d\mu\text{ and}\qquad & \qquad g(t,u) &= \frac{1}{2\pi} \int_{-\infty}^\infty e^{i\mu u} A(t,\mu) d\mu.
\end{align*}
\end{Proposition}
\begin{proof}
The result follows immediately from the Definitions \ref{def_locstat_cont} and \ref{def_locstat_cont2} using Plancherel's theorem.
%Parseval's identity, i.e. $\|A_N^0(t,\cdot)\|_{L^2}=\sqrt{2\pi}\|g_N^0(t,t-\cdot)\|_{L^2}$.		
\end{proof}

%Usually we will later on specify (locally stationary) models via the limiting kernel/transfer function. The $L^2$ continuity demanded both in the definition in the time and the frequency domain can then in many cases be obtained from continuity in the first argument by the following lemma.
The following lemma provides sufficient conditions for the continuity conditions on the mappings $t\mapsto A(t,\cdot)$ and $t\mapsto g(t,\cdot)$ from Definition \ref{def_locstat_cont} and \ref{def_locstat_cont2}.

\begin{Lemma}\label{lem:L2cont}
Let $A:\R\times\R\rightarrow\C$ be a function, which is continuous in the first argument such that for all $t\in\mathbb{R}$ there exists an $\epsilon_t>0$ and a real function $f_t\in  L^2(\R)$ such that $\abs{A(s,\cdot)}\leq f_t(\cdot)$ for all $s\in[t-\epsilon_t,t+\epsilon_t]$. Then, the mapping $\R \to L^2(\R)$, $t\mapsto A(t,\cdot)$ is continuous.
\end{Lemma}
\begin{proof}
Straightforward application of the dominated convergence theorem.
\end{proof}

%In principle it is possible to replace the L\'evy process by a process with weakly stationary uncorrelated increments and the ORM induced by the L\'evy process could be replaced by an arbitrary ORM. The resulting processes would be (locally) weakly stationary. However, to us it seems at the moment not worthwhile to pursue this any further.\\
%The next proposition is the continuous-time analogue of Proposition \ref{prop_1_1_ConvergenceInDistribution}. Note that the stationary and independent increments of the driving L\'evy process $L$ are essential to derive this result. Therefore, the orthogonal random measure in Definition \ref{def_locstat_cont2} has to be generated by a stochastic process with independent and stationary increments, i.e by a L\'evy process.

In principle it is possible to replace the L\'evy process by a process with weakly stationary uncorrelated increments and the ORM induced by the L\'evy process by an arbitrary ORM. The resulting processes would be (locally) weakly stationary. However, to us it seems at the moment not worthwhile to pursue this any further for the following reason.\\
To derive a continuous-time analogue of Proposition \ref{prop_1_1_ConvergenceInDistribution} the stationary and independent increments of the driving L\'evy process $L$ are essential. Therefore, also the orthogonal random measure in Definition \ref{def_locstat_cont2} has to be generated by a stochastic process on $\R$ with independent and stationary increments, i.e by a L\'evy process.\\
We note that this also ensures for all $\tilde{t}\in\R$ that the limiting process $Y_{\tilde{t}}(t)=\int_\R g(\tilde{t}, t-u) L(du)$ is strictly stationary.
The next proposition provides the aforementioned continuous-time analogue of Proposition \ref{prop_1_1_ConvergenceInDistribution}.

\begin{Proposition} 
Let $\{Y_N(t), t\in\R\}_{N\in\N}$ be a locally stationary process. Then, for fixed $t\in\R$ 
\begin{align*}
Y_N(t) \overset{d}{\underset{n\rightarrow\infty}{\longrightarrow}} \int_{\R} A(t,\mu) \Phi_L(d\mu) = \int_{\R} g(t,-u) L(du).
\end{align*}
\end{Proposition}
\begin{proof}
For $t\in\R$ we obtain, using a stationarity argument,%using similar arguments as in the proof of Proposition \ref{prop_1_1_ConvergenceInDistribution},
\begin{align*}
Y_N(t) = \int_{\R} e^{i\mu Nt} A_N^0(Nt,\mu) \Phi_L(d\mu) \overset{d}{=} \int_{\R} A_N^0(Nt,\mu) \Phi_L(d\mu).
\end{align*}
The remainder follows from the continuity of the stochastic integral in mean square and thus in distribution with respect to the integrand.
\end{proof}

\begin{Proposition}\label{prop_2___AsymptoticUncorrelated}
Let $\{Y_N(t), t\in\R\}_{N\in\N}$ be a locally stationary sequence and $t_1,t_2\in\R$ such that $t_1\neq t_2$. Then, $Y_N(t_1)$ and $Y_N(t_2)$ are asymptotically uncorrelated, i.e. $Cov(Y_N(t_1),Y_N(t_2))\rightarrow0$ as $N\rightarrow\infty$.
\end{Proposition}
The intuition behind this proposition is that the kernel functions $g_N^0(Nt,Nt-\cdot)$ are square integrable and therefore roughly vanish if the second argument tends to infinity. For $t_1\neq t_2$ the difference between $Nt_1$ and $Nt_2$ increases for $N\rightarrow\infty$. Therefore, for large $N$, the bulks of the kernels for $t_1$ and $t_2$ rest on far apart segments of the L\'evy process, which has independent increments.
\begin{proof} 	
Let $Y_N(t) = \int_{\R} g_N^0(Nt,Nt-u) L(du)$ be a sequence of locally stationary processes. Without loss of generality we assume that $t_1>t_2$ and set $h=t_1-t_2>0$. It is sufficient to show that for all $t_1\in\R$ and $\varepsilon>0$ there exists an $N_0\in\N$ such that for all $N>N_0$ 
\begin{align*}
\abs{Cov(Y_N(t_2),Y_N(t_2+h))} =\Sigma_L \abs{ \int_\R g_N^0(Nt_2,Nt_2-u) g_N^0(N(t_2+h),N(t_2+h)-u) du} <\varepsilon.
\end{align*}
Let $t\in\R$ and define $\mathcal{E}$ as the set of all elementary real functions in $L^2(\R)$, i.e.
\begin{align*}
\mathcal{E}=\left\{ f\in L^2(\R): f=\sum_{i=1}^n c_i  \mathbb{1}_{[a_i,b_i)}, n\in\N, c_i\in\R, -\infty<a_i<b_i<\infty, i=1,\ldots,n \right\}.
\end{align*}
Then for all $\eta>0$ there exists $N_1\in\N$ and elementary functions $\hat{g}(t_2,\cdot),\hat{g}(t_2+h,\cdot)\in\mathcal{E}$ such that
\begin{align*}
\norm{g(t_2,\cdot)-\hat{g}(t_2,\cdot)}_{L^2}&<\eta,\qquad & \qquad  \norm{g(t_2+h,\cdot)-\hat{g}(t_2+h,\cdot)}_{L^2}&<\eta,\\
\norm{g_N^0(Nt_2,\cdot)-g(t_2,\cdot)}_{L^2}&<\eta\text{ and}\qquad & \qquad \norm{g_N^0(N(t_2+h),\cdot)-g((t_2+h),\cdot)}_{L^2}&<\eta,
\end{align*} 
using \cite[Prop. 6.8]{R1988}. For the remainder of the proof, it will be assumed that $N>N_1$. Thus,
\begin{align*}	
\norm{g_N^0(N(t_2+h),\cdot)}_{L^2} \leq \eta + \norm{g(t_2+h,\cdot)}_{L^2}\text{ and}\qquad \norm{g_N^0(N(t_2),\cdot)}_{L^2} \leq \eta + \norm{g(t_2,\cdot)}_{L^2}.
\end{align*}
We define the constant $K= \eta+\max\left\{\norm{g(t_2,\cdot)}_{L^2} , \norm{g(t_2+h,\cdot)}_{L^2} \right\}<\infty$. Then, using the triangle and Cauchy-Schwartz's inequality shows
\begin{align*}
|Cov(Y_N(t_2),Y_N(t_2+h))| &= \Sigma_L \abs{\int_\R g_N^0(Nt_2,Nt_2-u) g_N^0(N(t_2+h),N(t_2+h)-u) du} \\
%		&\leq 	\Sigma_L \int_\R |g_N^0(Nt,Nt-u)| ~ |g_N^0(N(t+h),N(t+h)-u)| du \\
		%&\leq 	\Sigma_L \Big[ \int_\R |g_N^0(Nt,Nt-u)-g(t,Nt-u)| ~ |g_N^0(N(t+h),N(t+h)-u)| du \\
					%&\hspace{0.4cm}+ \int_\R |g(t,Nt-u)| ~ |g_N^0(N(t+h),N(t+h)-u)-g(t+h,N(t+h)-u)| du \\
					%&\hspace{0.4cm}+ \int_\R |g(t,Nt-u)| ~ |g(t+h,N(t+h)-u)| du \Big] \\
		%&\stackrel{\mathcal{E}}{\leq} \Sigma_L \Big[ \int_\R |g_N^0(Nt,Nt-u)-g(t,Nt-u)| ~ |g_N^0(N(t+h),N(t+h)-u)| du \\
					%&\hspace{0.4cm}+ \int_\R |g(t,Nt-u)| ~ |g_N^0(N(t+h),N(t+h)-u)-g(t+h,N(t+h)-u)| du \\
					%&\hspace{0.4cm}+ \int_\R |g(t,Nt-u)-\widetilde{g}(t,Nt-u)| ~ |g(t+h,N(t+h)-u)| du \\
					%&\hspace{0.4cm}+ \int_\R |\widetilde{g}(t,Nt-u)| ~ |g(t+h,N(t+h)-u)-\widetilde{g}(t+h,N(t+h)-u)| du \\
					%&\hspace{0.4cm}+ \int_\R |\widetilde{g}(t,Nt-u)| ~ |\widetilde{g}(t+h,N(t+h)-u)| du \Big] \\		
&\leq \Sigma_L \Big( \|g_N^0(Nt_2,\cdot)-g(t_2,\cdot)\|_{L^2} \|g_N^0(N(t_2+h),\cdot)\|_{L^2} \\
&\quad+ \|g(t_2,\cdot)\|_{L^2} \|g_N^0(N(t_2+h),\cdot)-g(t_2+h,\cdot)\|_{L^2} \\
&\quad+ \|g(t_2,\cdot)-\hat{g}(t_2,\cdot)\|_{L^2} \|g(t_2+h,\cdot)\|_{L^2} \\
&\quad+ \|\hat{g}(t_2,\cdot)\|_{L^2} \|g(t_2+h,\cdot)-\hat{g}(t_2+h,\cdot)\|_{L^2} \\
&\quad+ \int_\R |\hat{g}(t_2,Nt_2-u)| ~ |\hat{g}(t_2+h,N(t_2+h)-u)| du \Big) \\
&\leq \Sigma_L \Big( 4\eta K+ \int_\R |\hat{g}(t_2,Nt_2-u)| ~ |\hat{g}(t_2+h,N(t_2+h)-u)| du \Big),
\end{align*}
where the last integral tends to zero for $N\rightarrow\infty$ by using the dominated convergence theorem and noting that the elementary functions $\hat{g}$ have bounded support.
\end{proof}

\section{Classes of locally stationary processes in continuous-time}
\label{sec4}
%L\'evy-driven continuous-time ARMA (CARMA) processes constitute a broad and popular class of stochastic processes. 
In this section, we consider sequences of time-varying CARMA processes, for which we derive sufficient conditions for local stationarity.
\subsection{Locally stationary CAR(1) processes}
\label{sec4-1}
The simplest L\'evy-driven CARMA process is the L\'evy-driven CAR(1) or Ornstein-Uhlenbeck type process. \\
For a constant coefficient $a>0$ a CAR(1) process is the stationary solution to the stochastic differential equation $dY(t)=-aY(t)dt+L(dt)$,
which can be expressed as
\begin{align*}
Y(t)=\int_{-\infty}^t e^{-a (t-u)} L(du).
\end{align*}
We replace the constant $a$ by a time-varying function $a(t)$ and arrive at a so called time-varying CAR(1) process, which is given by
\begin{align*}
Y(t)=\int_{-\infty}^t e^{-\int_u^t a(s) ds} L(du).
\end{align*}
Additional rescaling results in a sequence of time-varying CAR(1) processes that could be locally stationary. We consider the sequence of stochastic processes $\{Y_N(t), t\in\R\}_{N\in\N}$ defined by
\begin{align}\label{locstat_CAR1}
		Y_N(t) 	&= \int_{-\infty}^{Nt} e^{-\int_u^{Nt} a(\frac{s}{N}) ds} L(du),
						%= \int_{\R} g_N^0(Nt,Nt-u) L(du) \\
						%&= \int_{\R} e^{i\mu Nt} A_N^0(Nt,\mu) \Phi_L(d\mu).
\end{align}
where $a:\R\rightarrow\R$ is a continuous coefficient function such that $u\mapsto e^{-\int_u^{Nt} a(\frac{s}{N}) ds}$ is square integrable for all $t\in\R$, $N\in\N$ and $L$ is a two-sided L\'evy process. Recall that the L\'evy process satisfies (\ref{ass_L}).
In view of Definition \ref{def_locstat_cont}, we obtain from (\ref{locstat_CAR1}) that				
\begin{align}
\begin{aligned}\label{equation:substitutionsNt}
g_N^0(Nt,Nt-u) 	&=  \mathbb{1}_{\{Nt-u\geq 0\}} e^{-\int_u^{Nt} a(\frac{s}{N}) ds} &&= \mathbb{1}_{\{Nt-u\geq 0\}} e^{-\int_{-(Nt-u)}^0 a(\frac{s}{N}+t) ds}\text{ and} \\
A_N^0(Nt,\mu)	&= \int_\R e^{-i\mu u} g_N^0(Nt,u) du &&= \int_\R e^{-i\mu v} \mathbb{1}_{\{v\geq 0\}} e^{-\int_{-v}^0 a(\frac{s}{N}+t) ds} dv.
\end{aligned}
\end{align}

\begin{Proposition} \label{prop_C1_C2_CAR1}
Let $\{Y_N(t), t\in\R\}_{N\in\N}$ be a sequence of time-varying CAR(1) processes as defined in (\ref{locstat_CAR1}). If 
\begin{enumerate}
\item[$(C1)$] $a(\cdot)$ is continuous and
\item[$(C2)$] for every $T\in\mathbb{R}^+$ there exists $\varepsilon_T>0$ such that $a(s)\geq \varepsilon_T$ for all $s\leq T$, %\\
%$\Leftrightarrow \mathfrak{Re}(\lambda(t))=\lambda(t)=-a(t)\leq -\varepsilon$ with $\lambda(t)$ eigenvalues of companion matrix $A(t)$ in the state space representation in Section~\ref{sec_2_2_Levy}.
\end{enumerate}
then $Y_N(t)$ is locally stationary, where the limiting kernel $g$ and limiting transfer function are given by
\begin{align*}
g(t,u)	=\mathbb{1}_{\{u\geq 0\}}e^{-a(t)u}\quad \text{ and}\qquad
A(t,\mu)=\int_\R e^{-i\mu u} \mathbb{1}_{\{u\geq 0\}}e^{-a(t)u} du.
\end{align*}
\end{Proposition}
\begin{proof}
For all $t\in\R$ it holds
\begin{align*}
\norm{g_N^0(Nt,\cdot)-g(t,\cdot)}_{L^2}^2 &=\norm{g_N^0(Nt,Nt-\cdot)-g(t,Nt-\cdot)}_{L^2}^2 \\
&=\int_\R \abs{ \mathbb{1}_{\{Nt-u\geq 0\}} e^{-\int_{-(Nt-u)}^0 a(\frac{s}{N}+t) ds} - \mathbb{1}_{\{Nt-u\geq 0\}}e^{-a(t)(Nt-u)} }^2 du \\
&= \int_\R \mathbb{1}_{\{u\leq 0\}} \abs{ e^{-\int_u^0 a(\frac{s}{N}+t) ds} - e^{a(t)u}}^2 du\underset{N\rightarrow\infty}{\longrightarrow}0,
\end{align*}
using the dominated convergence theorem. For the inner integral the continuity of $a$ on a compact set is sufficient for an application of the dominated convergence theorem. As majorant for the outer integal we consider $u\mapsto4\mathbb{1}_{\{u\leq 0\}} e^{2\varepsilon_t u}\in L^2$. The demanded $L^2$-continuity of the limiting kernel $g$ can be obtained similarly, using Lemma \ref{lem:L2cont}. 
\end{proof}

\begin{Remark} Condition $(C1)$ is intrinsically related to the continuity of the limiting kernel demanded in the definition of local stationarity. $(C2)$ is obviously satisfied if $a(\cdot)$ is bounded away from zero. However, as time goes to infinity $a(\cdot)$ may go to zero arbitrarily fast. The latter is clearly connected to the fact that our time-varying CAR processes are causal by definition. It may be possible to weaken $(C2)$ and allow $a(\cdot)$ also to approach $0$ as time goes to minus infinity. Then, the convergence to zero must be slow enough for all integrals to exist in $L^2$. Carrying this out in detail appears rather intricate and not of relevance for the applications of locally stationary CAR(1) processes.
%Conditions $(C1)$ and $(C2)$ in Proposition~\ref{prop_C1_C2_CAR1} are only sufficient and not necessary conditions for local stationarity, since $(C2)$ is used for finding a convergent majorant such that dominated convergence can be used.
\end{Remark}

\subsection{Time-varying CARMA(p,q) processes and time-varying state space models}
\label{sec4-2}
Consider $p,q\in\N$, where $p>q$. The formal differential equation for a time-varying L\'evy-driven CARMA($p,q$) process is given by
\begin{align*}
p(t,D)Y(t) &= q(t,D)DL(t),\text{ i.e.} \\
D^p Y(t) + a_1(t) D^{p-1}Y(t) + \ldots +a_p(t) Y(t) &= b_0(t) DL(t) + b_1(t) D^2L(t) +\ldots +b_q(t) D^{q+1}L(t),
\end{align*}
where $D$ denotes the differential operator with respect to time and $L(t)$ is a two sided L\'evy process satisfying (\ref{ass_L}). For continuous functions $a_i(t),b_i(t)$, $i=1,\ldots,p$, where $b_i(t)=0$ for all $i>q$, the polynomials
\begin{align}\label{eq_PQ_tv}
\begin{aligned}
		p(t,z) &= z^p+a_1(t)z^{p-1}+\ldots+a_{p-1}(t)z+a_p(t)\text{ and} \\
		q(t,z) &= b_0(t)+b_1(t)z+\ldots+b_{q-1}(t)z^{q-1}+b_q(t)z^q
\end{aligned}		
\end{align}
are called autoregressive (AR) and moving average (MA) polynomials. For a rigorous definition we interpret the differential equations to be equivalent to the state space representation 
\begin{align}\label{eq_statespace_tilde}
\begin{aligned}
Y(t)	&= \mathcal{B}(t)' \mathcal{X}(t), \text{ and}\\ 
d\mathcal{X}(t) &= \mathcal{A}(t) \mathcal{X}(t) dt + \mathcal{C} L(dt), 		\quad t\in\R  
\end{aligned}
\end{align}
with
\begin{align*}
\mathcal{A}(t)&=\left( \begin{array}{cccc}
0 & 1 & \ldots & 0 \\
\vdots & ~ & \ddots & \vdots \\
0 & ~ & ~ & 1 \\
-a_p(t) & -a_{p-1}(t) & \ldots & -a_1(t) \\
\end{array}\right)\in M_{p\times p}(\mathcal{R}[t])\text{ and}\\
\mathcal{B}(t)&=\left( \begin{array}{c}
b_0(t)  \\
 b_1(t)\\
\vdots \\
b_{p-1}(t)\\
\end{array}\right)\in M_{p\times 1}(\mathcal{R}[t]), \qquad 
\mathcal{C}=\left( \begin{array}{c}
0  \\
\vdots\\
0 \\
1\\
\end{array}\right)\in M_{p\times 1}(\R),
\end{align*}
where $\mathcal{R}[t]$ denotes the ring of continuous functions in $t$ from $\R$ to $\R$.

\noindent It is obvious that \eqref{eq_statespace_tilde} has a unique solution when one fixes the value $X(t_0)$ at some point $t_0\in\R$. For a Brownian motion as driving noise such equations were investigated in \cite[Section 2.1.1.]{S2010}.\\
Provided the integrals exist in $L^2$, it can be shown that a solution is given by
\begin{align}\label{sol_tvCARMA}
\mathcal{X}(t) = \int_{-\infty}^t \Psi(t,s) \mathcal{C} L(ds) \quad \text{and} \quad Y(t) = \mathcal{B}(t)' \int_{-\infty}^t \Psi(t,s) \mathcal{C} L(ds),
\end{align}
 where $\Psi(t,t_0)$ is the unique matrix solution of the homogeneous initial value problem (IVP) $\frac{d}{dt}\Psi(t,t_0)=\mathcal{A}(t)\Psi(t,t_0)$, where $\Psi(t_0,t_0)=\mathbf{1}_p$ for all $t>t_0$ (see \cite[Section 3 and 4]{B1970}). The transition matrix $\Psi$ satisfies $\Psi(t,t_0)=\Psi(t,u)\Psi(u,t_0)$ for all $t>u>t_0$ (see \cite[Section 4, Theorem 2]{B1970}). In particular, the integrals in (\ref{sol_tvCARMA}) are well-defined (see Section \ref{sec2-2}), if there exist $\gamma,\lambda>0$, such that
\begin{align*}
\norm{ \Psi(t,t_0)} \leq \gamma e^{-\lambda(t-t_0)} 
\qquad \text{ for all } t,t_0 \text{ with } t\geq t_0.
\end{align*}
This condition corresponds to uniform exponential stability of the state space model in (\ref{eq_statespace_tilde}) and will be explained more detailed in Section \ref{sec4-3}.\\
The usual integral representation of stationary causal CARMA processes motivates the following definition.
\begin{Definition}\label{def_tvCARMA}
A solution $\{Y(t)\}_{t\in\R}$ of the observation and state equations (\ref{eq_statespace_tilde}) in the form \eqref{sol_tvCARMA} is called a time-varying L\'evy-driven CARMA(p,q) process (tvCARMA(p,q)). 
\end{Definition}

\noindent For some initial time $t_0\in\R$ the process satisfies the relation (see \cite[Section~2.1.1.]{S2010})
\begin{align}\label{sol_tvCARMA_t0}
\mathcal{X}(t) = \Psi(t,t_0)\left(\mathcal{X}(t_0)+\int_{t_0}^t \Psi(u,t_0)^{-1} \mathcal{C} L(du)\right).
\end{align}
From  \cite[Remark 2]{BS2011b} it follows that if for all $t,t_0\in\R$ and $t>t_0$
\begin{align}\label{ass_commutative}
\mathcal{A}(t)\int_{t_0}^t \mathcal{A}(s)ds = \int_{t_0}^t \mathcal{A}(s)ds \mathcal{A}(t),
\end{align}
then $\Psi(t,t_0)=e^{\int_{t_0}^t \mathcal{A}(s)ds}$. \\
If assumption (\ref{ass_commutative}) does not hold, $\Psi(t,t_0)$ can be expressed by the Peano-Baker series (see \cite[Section 2]{BS2011b})
\begin{align*}
\Psi(t,t_0) 	&= \mathbf{1}_p + \int_{t_0}^t \mathcal{A}(\tau_1) d\tau_1 + \int_{t_0}^t \mathcal{A}(\tau_1) \int_{t_0}^{\tau_1} \mathcal{A}(\tau_2) d\tau_2 d\tau_1 + \ldots = \sum_{n=0}^\infty \mathcal{I}_n(t),
\end{align*}
where $\mathcal{I}_0(t)= \mathbf{1}_p$ and $\mathcal{I}_n(t)= \int_{t_0}^t \mathcal{A}(\tau_1) \int_{t_0}^{\tau_1} \mathcal{A}(\tau_2) \cdots \int_{t_0}^{\tau_{n-1}} \mathcal{A}(\tau_n) d\tau_n \ldots d\tau_2 d\tau_1$. \\% with $\frac{d}{dt}\mathcal{I}_{n+1}(t)=A(t)\mathcal{I}_n(t)$.\\
If the commutativity assumption (\ref{ass_commutative}) holds, the equations (\ref{sol_tvCARMA}) and (\ref{sol_tvCARMA_t0}) simplify to
\begin{align*}
\mathcal{X}(t) 	&= e^{\int_{t_0}^t \mathcal{A}(s)ds}\mathcal{X}(t_0)+\int_{t_0}^t e^{\int_u^t \mathcal{A}(s)ds} \mathcal{C} L(du) &&= \int_{-\infty}^t e^{\int_u^t \mathcal{A}(s)ds} \mathcal{C} L(du)\text{ and}  \\
Y(t) 	&= \mathcal{B}(t)' e^{\int_{t_0}^t \mathcal{A}(s)ds}\mathcal{X}(t_0)+\int_{t_0}^t \mathcal{B}(t)' e^{\int_u^t \mathcal{A}(s)ds} \mathcal{C} L(du)&&= \int_{-\infty}^t \mathcal{B}(t)' e^{\int_u^t \mathcal{A}(s)ds} \mathcal{C} L(du)
\end{align*}
for $t,t_0\in\R$, where $t>t_0$.

\begin{Remark}\label{rem_commutative}
If $\mathcal{A}(s)$ and $\mathcal{A}(t)$ commute, i.e. $[\mathcal{A}(s),\mathcal{A}(t)]=0$ for all $s,t\in\R$, then the commutativity assumption (\ref{ass_commutative}) holds. However, the matrices $\mathcal{A}(t)$, $t\in\R$ are in companion form and are not in general commutative (see also Proposition \ref{Proposition:commutativitycompanion}). For further insight into the commutativity of some matrices $\mathcal{A}(t)$ and $\int_{t_0}^t \mathcal{A}(s)ds$ as well as $\mathcal{A}(s)$ and $\mathcal{A}(t)$, we refer to \cite[Exercise 4.8]{R1996} and \cite{WS1976}.
\end{Remark}

The previous remark shows that, when considering time-varying CARMA(p,q) processes, it is in general not possible to describe the solution of the state space equations explicitly in form of a matrix exponential. Instead one has to use the Peano-Baker series.\\
In \cite[Corollary 3.4]{SS2012} it is proved that, in the time-invariant case, the class of CARMA processes is equivalent to the class of continuous-time state space models. This motivates looking at time-varying state space processes. We consider the observation and state equations 
\begin{align}\label{eq_statespace_notilde}
\begin{aligned}
Y(t)	&= B(t)' X(t) \text{ and} \\
dX(t) &= A(t)X(t)dt + C(t) L(dt),
\end{aligned}
\end{align}
where $t\in\R$, $A(t)\in M_{p\times p}(\mathcal{R}[t])$ and $B(t),C(t)\in M_{p\times1}(\mathcal{R}[t])$ are arbitrary continuous coefficient functions and $L$ is a two-sided L\'evy process satisfying (\ref{ass_L}).\\
Now, the representation of a time-varying CARMA processes as given in (\ref{sol_tvCARMA}) can be adapted to (general) state space processes. Provided the integrals exist in $L^2$, it can be shown that a solution of (\ref{eq_statespace_notilde}) is given by
\begin{align}\label{sol_tvStateSpace}
X(t) = \int_{-\infty}^t \Psi(t,u) C(u) L(du) \quad \text{and} \quad Y(t) = B(t) \int_{-\infty}^t \Psi(t,u) C(u) L(du),
\end{align}
where $\Psi(t,t_0)$ is the unique matrix solution of the IVP $\frac{d}{dt}\Psi(t,t_0)=A(t)\Psi(t,t_0)$, $\Psi(t_0,t_0)=\mathbf{1}_p$ for $t>t_0$. In particular, in the representation of (\ref{sol_tvStateSpace}), the integrals are well-defined, if $C(\cdot)\in L^2(\R)$ and there exist $\gamma,\lambda>0$, such that
\begin{align*}
\norm{\Psi(t,t_0)} \leq \gamma e^{-\lambda(t-t_0)} 
\qquad \text{for all }t\geq t_0 \text{ (uniform exponential stability)}.
\end{align*}
For some initial time $t_0\in\R$ the process satisfies the relation
\begin{align}\label{sol_tvStateSpace_t0}
X(t) = \Psi(t,t_0)\left(X(t_0)+\int_{t_0}^t \Psi(u,t_0)^{-1} C(u) L(du)\right).
\end{align}
Finally, we define

\begin{Definition} \label{def_tvStateSpace}
A solution $\{Y(t)\}_{t\in\R}$ of the observation and state equations (\ref{eq_statespace_notilde}) in the form (\ref{sol_tvStateSpace}) is called a time-varying L\'evy-driven state space process.
\end{Definition}

The natural question arises, whether all time-varying state space processes are tvCARMA processes, as in the time-invariant case. A comprehensive investigation of this question seems beyond the scope of this work. Below we present a result indicating that this is probably not the case in general (definitely not when allowing the coefficient functions to have a discontinuity). Moreover, we give sufficient conditions for a positive answer.

\begin{Proposition}
The class of time-varying L\'evy-driven CARMA models (\ref{eq_statespace_tilde}) and time-varying L\'evy-driven state space models (\ref{eq_statespace_notilde}) with not necessarily continuous coefficient functions do not coincide in general.
\end{Proposition}

\begin{proof}
Consider a two dimensional time-varying state space model as defined in (\ref{eq_statespace_notilde}) with a structural break at $t=1$. As coefficient functions we consider the step functions
\begin{align}\label{equation:coefficientfunctionstimevarying1}
B(t)=\begin{cases} B_1 &\text{if~} t\leq1 \\ B_2 & \text{if~} t>1 \end{cases}, \quad
A(t)=\begin{cases} A_1 &\text{if~} t\leq1 \\ A_2 & \text{if~} t>1 \end{cases}\text{ and} \quad
C(t)=\begin{cases} C_1 &\text{if~} t\leq1 \\ C_2 & \text{if~} t>1 \end{cases},
\end{align}
which satisfy the uniform exponential stability assumption for the solution of (\ref{eq_statespace_notilde}).\\
We assume that the system is in the form of a CARMA process for $t\leq1$ and assume (for contradiction) that there exists an equivalent CARMA model as defined in (\ref{eq_statespace_tilde}) for all $t\in\R$. Then, the CARMA model shows the same structural resemblance as the corresponding state space model. In the following we denote the coefficients of the CARMA model by $\mathcal{B}(t),\mathcal{A}(t)$ and $\mathcal{C}(t)$. Using the same notation as in (\ref{equation:coefficientfunctionstimevarying1}) we obtain
\begin{align}
\begin{aligned}	
&\mathcal{B}_1  	= B_1, \qquad
&&\mathcal{A}_1 	= A_1, \qquad
&&\mathcal{C}_1  = C_1, \\
&\mathcal{B}_2 	= \begin{pmatrix} \ast \\ \ast \end{pmatrix}, \qquad
&&\mathcal{A}_2 	= \begin{pmatrix} 0 & 1 \\ \ast & \ast \end{pmatrix}\quad \text{ and}\quad
&&\mathcal{C}_2 	= \begin{pmatrix} 0 \\ 1 \end{pmatrix}.
\label{eq_StructureBreak}
\end{aligned}
\end{align}
Since the structural break divides the model in two separate linear models, the CARMA representations $(\mathcal{B}_1,\mathcal{A}_1,\mathcal{C}_1)$ and $(\mathcal{B}_2,\mathcal{A}_2,\mathcal{C}_2)$ are unique. From the proof of \cite[Theorem 3.3]{SS2012} we obtain that $\mathcal{B}_2 e^{\mathcal{A}_2 (t-u)} \mathcal{C}_2 = B_2 e^{A_2 (t-u)} C_2$ for all $t>1$ and $u\in(1,t]$.\\
On the one hand, we have %it holds for the process $\{Y(t)\}_{t\in\R}$ that
\begin{align}
\begin{aligned}\label{equation:representationY1}
Y(t)	&= 	\mathbb{1}_{\{t\leq1\}} \left( \int_{-\infty}^t B(t)' \Psi_A(t,u) C(u) L(du) \right) \\
	&\quad+	\mathbb{1}_{\{t>1\}} \left( B(t)' \Psi_A(t,1) X(1) + B(t)' \Psi_A(t,1) \int_1^t \Psi_A(u,1)^{-1} C(u) L(du) \right) \\
	&= 	\mathbb{1}_{\{t\leq1\}} \left( \int_{-\infty}^t B_1' e^{A_1(t-u)} C_1 L(du) \right)\\
	&\quad+	\mathbb{1}_{\{t>1\}} \left( B_2' e^{A_2(t-1)} X(1) + \int_1^t B_2' e^{A_2(t-u)} C_2 L(du) \right),
\end{aligned}
\end{align}
where $\Psi_A(s,s_0)$ denotes the solution of the aforementioned IVP with respect to $A$. On the other hand, $Y(t)$ can be written as
\begin{align}
\begin{aligned}\label{equation:representationY2}
Y(t)&= \mathbb{1}_{\{t\leq1\}} \left( \int_{-\infty}^t \mathcal{B}_1' e^{\mathcal{A}_1(t-u)} \mathcal{C}_1 L(du) \right) \\
&\quad+\mathbb{1}_{\{t>1\}} \left( \mathcal{B}_2' e^{\mathcal{A}_2(t-1)} \mathcal{X}(1) + \int_1^t \mathcal{B}_2' e^{\mathcal{A}_2(t-u)} \mathcal{C}_2 L(du) \right). 
\end{aligned}
\end{align}
From (\ref{eq_StructureBreak}) it follows that
\begin{align*}
X(1)	&= \int_{-\infty}^1 e^{A_1(1-u)} C_1 L(du) =\int_{-\infty}^1 e^{\mathcal{A}_1(1-u)} \mathcal{C}_1 L(du)=\mathcal{X}(1).
\end{align*}
Thus, combining (\ref{equation:representationY1}) and (\ref{equation:representationY2}), using (\ref{eq_StructureBreak}) and the independent increments of the L\'evy process, the equality $B_2' e^{A_2(t-1)} X(1)=\mathcal{B}_2' e^{\mathcal{A}_2(t-1)} X(1)$ has to hold almost surely for all $t>1$. Therefore, for almost all $x$ in the support of $X(1)$ we obtain
\begin{align}\label{eq_proof_ToContradict}
B_2' e^{A_2(t-1)} x=\mathcal{B}_2' e^{\mathcal{A}_2(t-1)} x.
\end{align}
In the sequel we give a particular L\'evy process and coefficient functions that lead to a contradiction in (\ref{eq_proof_ToContradict}).\\
Assume that the L\'evy process is a Brownian motion. Thus, it has the triplet $(0,\Sigma,0)$ for some $\Sigma>0$. From \cite{S2014} it follows that $X(t)$ is a L\'evy process with triplet $(0,\Sigma_X^t,0)$, where $\Sigma_X^1= \int_{-\infty}^1 e^{A_1(1-u)} C_1 \Sigma C_1' e^{A_1'(1-u)} du = \Sigma \int_0^\infty e^{A_1 u} C_1 C_1' e^{A_1' u} du \in M_{2\times 2}(\R)$. 
% From the construction of Sigma_X we geht directly that it is symmetric and positive semidefinite. We need regularity for the following reason: If Sigma_X would not be regular, the process would still be Gaussian, since it is just needed that Sigma is symmetric positiv semi-definite (see e.g. Levy khintchin formula). But we need regularity since then the density is absolut continuous with respect to the lebesgue measure which would give that the process has support on whole \R which we need. Since then Sigma_X is invertible, it is also positive definite
The regularity of $\Sigma_X^1$ can be shown by investigating $Im(\Sigma_X^1)$, where $Im(D)=\{Dx: x\in\R^d\}$ denotes the image of a matrix $D\in M_{d\times d}(\R)$. Using \cite[Lemma 12.6.2]{B2009} (see also \cite[p. 54]{SS2012}) we obtain 
\begin{align*}
Im\left(\int_0^\infty e^{A_1 u} C_1 C_1' e^{A_1' u} du\right) = Im\left([C_1\ \ A_1 C_1\ \ \cdots\ \ A_1^{p-1}C_1]\right).
\end{align*}
Therefore, in our setting, it is sufficient to find $A_1,C_1$ such that $[C_1 \ A_1 C_1]$ is regular, which also implies that $\Sigma_X^1$ is positive definite. Then, $X(1)$ has characteristic function $E[e^{i\left\langle  z,X(1) \right\rangle }] = e^{-\tfrac{1}{2}\left\langle  z,\Sigma_X^1 z \right\rangle}$, which corresponds to a two dimensional $N(0,\Sigma_X^1)$ distributed random variable, having positive density for all values $x\in\R^2$. To contradict (\ref{eq_proof_ToContradict}), it is enough to show that for some $t>1$
\begin{align*}
B_2' e^{A_2(t-1)} x \neq \mathcal{B}_2' e^{\mathcal{A}_2(t-1)} x \qquad \text{ for all } x\in I \text{, where } I\subset\R^2 \text{ with } \lambda(I)>0.
\end{align*}
We define
\begin{align*}
		&B_1  	= \begin{pmatrix} 1 \\ 2 \end{pmatrix},\ B_2 	= \begin{pmatrix} 1 \\ 1 \end{pmatrix},\ \mathcal{B}_2 = \begin{pmatrix} 5 \\ 2 \end{pmatrix}, \quad
		A_1 	= \begin{pmatrix} 0	& 1 \\ 1 &1 \end{pmatrix}, \ A_2 	= \begin{pmatrix} -2 & 0 \\ 0 & -3 \end{pmatrix},\ \mathcal{A}_2 = \begin{pmatrix} 0 & 1 \\ -6	& -5 \end{pmatrix}\text{ and} \\
		&C_1  = \begin{pmatrix} 0 \\ 1 \end{pmatrix},\ C_2 =  \begin{pmatrix} 1 \\ 1 \end{pmatrix},\ \mathcal{C}_2 = \begin{pmatrix} 0 \\ 1 \end{pmatrix}.
\end{align*}
From this we obtain that the CARMA model has the same transfer function as the state space model, since $B_2' (z\mathbf{1}_2-A_2)^{-1} C_2 = \frac{2z+5}{z^2+5z+6}=\mathcal{B}_2' (z-\mathbf{1}_2\mathcal{A}_2)^{-1} \mathcal{C}_2$. Moreover, $[C_1\ \ A_1 C_1]=\begin{pmatrix} 0 & 1 \\ 1	& 1\end{pmatrix}$ is regular. Given a vector $x=(x_1, x_2)'$ it is left to investigate $\mathcal{B}_2' e^{\mathcal{A}_2(t-1)} x - B_2' e^{A_2(t-1)} x$. For a matrix $D\in\C^{2\times 2}$ with eigenvalues $\sigma(A)=\{\mu,\lambda\}$, \cite[Proposition 11.3.2]{B2009} gives that
\begin{align*}
e^D=\begin{cases}  e^\lambda ((1-\lambda)\mathbf{1}_2+D) & \text{if } \mu=\lambda \\ 
\frac{\mu e^\lambda - \lambda e^\mu}{\mu-\lambda} \mathbf{1}_2 + \frac{e^\mu - e^\lambda}{\mu-\lambda} D & \text{if } \mu\neq\lambda. \end{cases}
\end{align*}
Since $\sigma(\mathcal{A}_2(t-1))=\{-2(t-1),-3(t-1)\}$, we obtain
\begin{align*}
\mathcal{B}_2' e^{\mathcal{A}_2(t-1)} x - B_2' e^{A_2(t-1)} x &=	\begin{pmatrix} 5 & 2 \end{pmatrix}\Bigg(\frac{-3(t-1) e^{-2(t-1)} - (-2)(t-1) e^{-3(t-1)}}{-3(t-1)-(-2)(t-1)} \begin{pmatrix} 1 & 0 \\ 0 & 1 \end{pmatrix} \\
&\qquad\qquad+ \frac{e^{-3(t-1)} - e^{-2(t-1)}}{-3(t-1)-(-2)(t-1)} \begin{pmatrix} 0 & 1 \\ -6 & -5 \end{pmatrix} (t-1) \Bigg)\begin{pmatrix} x_1 \\ x_2 \end{pmatrix} \\
&\qquad\qquad- \begin{pmatrix} 1 & 1 \end{pmatrix}\Bigg(\begin{pmatrix} e^{-2(t-1)} & 0 \\ 0 & e^{-3(t-1)} \end{pmatrix}\Bigg)\begin{pmatrix} x_1 \\ x_2 \end{pmatrix} \\
		%&=	\Bigg[
				%\left(3 e^{-2(t-1)} - 2 e^{-3(t-1)}\right) \begin{pmatrix} 5 & 2 \end{pmatrix}
				%+ \left(e^{-2(t-1)} - e^{-3(t-1)}\right) \begin{pmatrix} -12 & -5  \end{pmatrix}
				%\Bigg]
				%\begin{pmatrix} x_1 \\ x_2 \end{pmatrix} \\
		%&\qquad- \Bigg[
				%\begin{pmatrix} e^{-2(t-1)} & e^{-3(t-1)} \end{pmatrix}
				%\Bigg]
				%\begin{pmatrix} x_1 \\ x_2 \end{pmatrix} \\
		%&= \left(3 e^{-2(t-1)} - 2 e^{-3(t-1)}\right) \left(5 x_1 + 2 x_2\right)
				%+ \left(e^{-2(t-1)} - e^{-3(t-1)}\right) \left(-12 x_1 - 5 x_2\right) \\
		%&\qquad- \left( e^{-2(t-1)} x_1 + e^{-3(t-1)} x_2 \right) \\
&=	2x_1 \left( e^{-2(t-1)} + e^{-3(t-1)}\right) +x_2 \left( e^{-2(t-1)} \right) >0
\end{align*}
for all $x\in I=\{x\in\R^2 : x_1>0,x_2>0\}$ and $t>1$.
%This shows that --- when considering step functions $A(t)$ --- the class of time-varying state space models is not equivalent to the class of time-varying CARMA processes. \\
\end{proof}

Under more rigorous conditions on the coefficient functions, the concept of controllability from linear system theory allows for a characterization for special \textit{canonical} forms, which occur in the state space representation of CARMA processes ($\mathcal{A}$ is in companion matrix form). The following results summarize the key aspects of this characterization, which is mainly based on \cite{S1966}, but also \cite{B1987, RR1969, RR1971}.

\begin{Definition}[{\cite[Chapter 9 and 10]{R1996}}]
Let $Y(t)$ be a state space model as defined in (\ref{eq_statespace_notilde}), where $A(t)$ is $(p-1)$-times continuously differentiable and $C(t)$ $p$-times. We define the controllability matrix $W_p(t)$, as 
\begin{align*}
W_p(t)&=[K_0(t)\  K_1(t)\ \cdots \ K_{p-1}(t)], \text{ where}\\
K_0(t)&=C(t), \quad	K_{i+1}(t) = - A(t) K_i(t) + \tfrac{d}{dt}K_i(t), \quad i=1,\ldots,p-2.
\end{align*}
Then, the state process $X(t)$ is called 
\begin{enumerate}[label={(\alph*)}]
\item controllable on $[t_0,t_1]$, $t_0<t_1$, if there exists $t\in[t_0,t_1]$ with $Rank(W_p(t))=p$ and
\item instantaneously controllable, if $Rank(W_p(t))=p$ for all $t\in\R$.
\end{enumerate}
\end{Definition}

\begin{Proposition}[{\cite[Theorem 1]{S1966}}]\label{proposition:equivalencestatespacecanonical}
Consider a state space process satisfying (\ref{eq_statespace_notilde}) such that $A$ is $(p-1)$-times continuously differentiable and $C$ $p$-times. Then, it is equivalent to a CARMA process satisfying (\ref{eq_statespace_tilde}) if and only if it is instantaneously controllable. Equivalence means that there exists a regular matrix $T(t)\in M_{n\times n}(\mathcal{R}[t])$, which is continuously differentiable and satisfies
\begin{align*}
\mathcal{X}(t) = T(t) X(t)
\end{align*}
almost surely. The relationship between both systems is given by 
$T(t) = \mathcal{W}_p(t) W_p(t)^{-1}$, $\mathcal{A}(t) = \left( T(t) A(t) + \tfrac{d}{dt}T(t) \right) T(t)^{-1}$ and $\mathcal{C} = T(t) C(t)$, where $\mathcal{W}_p(t)$ and $W_p(t)$ are the controllability matrices of the state space model and the CARMA process.
\end{Proposition}

\begin{Corollary}
The class of time-varying L\'evy-driven state space models as defined in (\ref{eq_statespace_notilde}) with $(p-1)$-times continuously differentiable coefficient functions $A$, $p$-times continuously differentiable coefficient functions $C$ and controllability matrices $W_p(t)$ that have rank $p$ everywhere, is equivalent to the class of time-varying CARMA(p,q) processes as defined in (\ref{eq_statespace_tilde}) with $(p-1)$-times continuously differentiable coefficient functions $A$ and controllability matrices $W_p(t)$ that have rank $p$ everywhere.
%If the coefficient function $A$ is $(p-1)$-times continuously differentiable, $C$ $p$-times and the controllability matrix $W_p(t)$ of the L\'evy driven time-varying state space model in (\ref{eq_statespace_notilde}) has rank $p$ everywhere, then the class of L\'evy driven time-varying state space models in (\ref{eq_statespace_notilde}) is equivalent to the class of time-varying CARMA(p,q) processes.
\end{Corollary}

\begin{proof}
Any time-varying CARMA(p,q) process is obviously also a time-varying state space process. On the contrary, let $Y(t)$ be a time-varying state space process defined by (\ref{eq_statespace_notilde}), which is instantaneously controllable with controllability matrix $W_p(t)$. Then, due to Proposition \ref{proposition:equivalencestatespacecanonical}, the state system $dX(t) =  A(t)X(t)dt + C(t) L(dt)$ is equivalent to the CARMA system $d\mathcal{X}(t) = \mathcal{A}(t)\mathcal{X}(t)dt+\mathcal{C} L(dt)$ with 
\begin{align*}
\mathcal{X}(t) = T(t) X(t), \quad \mathcal{C} = T(t) C(t) \quad \text{and} \quad \mathcal{A}(t) = \left( T(t) A(t) + \tfrac{d}{dt}T(t) \right) T(t)^{-1},
\end{align*}
where $T(t) = \mathcal{W}_p(t) W_p(t)^{-1}$ is regular. Thus
\begin{align*}
Y(t)&= B(t)' X(t) = B(t)' T(t)^{-1} \mathcal{X}(t) = \mathcal{B}(t)' \mathcal{X}(t)\quad \text{and}\\
		d\mathcal{X}(t) &= \mathcal{A}(t)\mathcal{X}(t)dt+\mathcal{C} L(dt),
\end{align*}
which is a representation for $Y(t)$ as a time-varying CARMA(p,q) process in (\ref{eq_statespace_tilde}).
\end{proof}

\subsection{Locally stationary linear state space models - Peano-Baker series}\label{sec4-3}

We investigate sufficient conditions for sequences of time-varying state space processes, which obviously also includes sequences of time-varying CARMA processes, to be locally stationary. \\
Let $\{Y_N(t), t\in\R\}_{N\in\N}$ be a sequence of time-varying linear state space processes defined by%of dimension $p$, which is defined by its state space representation
\begin{align*} 
Y_N(t) 	&= B(t)' X_N(Nt)\quad \text{and} \\
X_N(Nt) &= \Psi_N(Nt,0) \int_{-\infty}^{Nt} \Psi_N(u,0)^{-1} C(\tfrac{u}{N}) L(du),
\end{align*}
where $\Psi_N(s,s_0)$ is the solution of the matrix differential equation
\begin{align*}
\Psi_N(s_0,s_0) &= \mathbf{1}_p \\
\tfrac{d}{ds}\Psi_N(s,s_0) &= A(\tfrac{s}{N}) \Psi_N(s,s_0), \quad \text{for all } s,s_0\in\R \text{ with } s>s_0,
\end{align*}
which can be expressed as (see \cite{BS2011b} Section 2)
\begin{align*}
\Psi_N(s,s_0) 	= \mathbf{1}_p + \int_{s_0}^s A(\tfrac{\tau_1}{N}) d\tau_1 + \int_{s_0}^s A(\tfrac{\tau_1}{N}) \int_{s_0}^{\tau_1} A(\tfrac{\tau_2}{N}) d\tau_2 d\tau_1 + \ldots \ .
\end{align*}
The substitution $s\mapsto s+Nt$ in (\ref{equation:substitutionsNt}), is necessary to achieve a dependence of the kernel function $g_N^0(t,\cdot)$ on $Nt-u$. Therefore, we define $\widetilde{\Psi}_{N,t}(0,-(Nt-u))$ for a fixed point $t\in\R$ as the solution of the matrix differential equation
\begin{align*}
\widetilde{\Psi}_{N,t}(s_0,s_0) &= \mathbf{1}_p \\
\tfrac{d}{ds}\widetilde{\Psi}_{N,t}(s,s_0) &= A(\tfrac{s}{N}+t) \widetilde{\Psi}_{N,t}(s,s_0), \quad \text{for all } s,s_0\in\R \text{ with } s>s_0,
\end{align*}
which can again be expressed as
\begin{align*}
\widetilde{\Psi}_{N,t}(s,s_0)= \mathbf{1}_p + \int_{s_0}^s A(\tfrac{\tau_1}{N}+t) d\tau_1 + \int_{s_0}^s A(\tfrac{\tau_1}{N}+t) \int_{s_0}^{\tau_1} A(\tfrac{\tau_2}{N}+t) d\tau_2 d\tau_1 + \ldots
\end{align*}
From \cite[Theorem 4.2.]{B1970} we obtain $\Psi_N(Nt,0)\Psi_N(u,0)^{-1} = \Psi_N(Nt,0)\Psi_N(0,u) = \Psi_N(Nt,u)$. Since,
\begin{align*}
\Psi_N(Nt,u) &= \mathbf{1}_p + \int_{u}^{Nt} A(\tfrac{\tau_1}{N}) d\tau_1 + \int_{u}^{Nt} A(\tfrac{\tau_1}{N}) \int_{u}^{\tau_1} A(\tfrac{\tau_2}{N}) d\tau_2 d\tau_1 + \ldots \\
&= \mathbf{1}_p + \int_{u-Nt}^0 A(\tfrac{\tau_1}{N}+t) d\tau_1 + \int_{u-Nt}^0 A(\tfrac{\tau_1}{N}+t) \int_{u-Nt}^{\tau_1} A(\tfrac{\tau_2}{N}+t) d\tau_2 d\tau_1 + \ldots \\
&= \widetilde{\Psi}_{N,t}(0,-(Nt-u)),
\end{align*}
we neglect the superscript tilde and define a sequence of time-varying linear state space processes as follows.

\begin{Definition}\label{def_sequencetimevarying_StateSpace_PBS}
A sequence of time-varying linear state space processes $\{Y_N(t), t\in\R\}_{N\in\N}$ is defined as
\begin{align}\label{locstat_CARMA_PBS}
Y_N(t) = \int_\R \mathbb{1}_{\{Nt-u\geq 0\}} B(t)'  \Psi_{N,t}^0(0,-(Nt-u)) C(\tfrac{u}{N}) L(du) %\\
%= \int_\R g_N^0(Nt,Nt-u) L(du)
\end{align}
with (limiting) kernel function (in view of Definition \ref{def_locstat_cont})						
\begin{align*}
g_N^0(Nt,Nt-u) 	&= \mathbb{1}_{\{Nt-u\geq 0\}} B(t)' \Psi_{N,t}^0(0,-(Nt-u)) C(\tfrac{-(Nt-u)}{N}+t) \quad\text{and}\\
g(t,Nt-u)	&= \mathbb{1}_{\{Nt-u\geq 0\}} B(t)' \Psi_t(0,-(Nt-u)) C(t),
\end{align*}
where $\Psi_{N,t}^0(0,-(Nt-u))$ and $\Psi_t(0,-(Nt-u))$ are the solutions of the matrix differential equations
\begin{align}
\begin{aligned}\label{equ_IVP_PBS}
\Psi_{N,t}^0(s_0,s_0) &= \mathbf{1}_p,  \quad &&\tfrac{d}{ds} \Psi_{N,t}^0(s,s_0) = A(\tfrac{s}{N}+t) \Psi_{N,t}^0(s,s_0),\\
\Psi_t(s_0,s_0)&= \mathbf{1}_p\quad \text{and} \quad && \tfrac{d}{ds} \Psi_t(s,s_0)= A(t) \Psi_t(s,s_0)		
\end{aligned}
\end{align}
for $s>s_0$.
\end{Definition}
\noindent Using the Peano-Baker series, if necessary, the solutions of the above matrix differential equations are given by $\Psi_t(s,s_0)=e^{A(t)(s-s_0)}$ and
\begin{align*}
\Psi_{N,t}^0(s,s_0) = \mathbf{1}_p+\int_{s_0}^s A(\tfrac{\tau_1}{N}+t) d\tau_1+ \int_{s_0}^s A(\tfrac{\tau_1}{N}+t) \int_{s_0}^{\tau_1} A(\tfrac{\tau_2}{N}+t) d\tau_2 d\tau_1 + \ldots.
\end{align*}

\begin{Proposition}\label{prop_C1_C2_C3_StateSpace_PeanoBaker}
Let $\{Y_N(t), t\in\R\}_{N\in\N}$ be a sequence of time-varying state space processes as in Definition \ref{def_sequencetimevarying_StateSpace_PBS}. If
\begin{enumerate}
\item[$(C1)$] the coefficient functions $A(\cdot)$, $B(\cdot)$ and $C(\cdot)$ are continuous,
\item[$(C2)$] $\norm{B(s)}<\infty$ for all $s\in\R$, $\sup_{s\in\R}\|C(s)\|<\infty$ and
\item[$(C3)$] $\norm{ \Psi_{N,t}^0(0,u)}  \leq F_t(u)$ for some real function $F_t\in L^2((-\infty,0])$ for all $N\in\N$ and $t\in\R$,
\end{enumerate}
then $Y_N(t)$ is locally stationary.
\end{Proposition}

\begin{proof}
Consider $Y_N(t)$, $g_N^0$, $g$, $\Psi_{N,t}^0$ and $\Psi_t$ as defined above. For fixed $u,t\in\R$ it holds
\begin{align*}
\abs{g_N^0(Nt,-u)-g(t,-u)}&= \mathbb{1}_{\{u\leq 0\}} \Big| B(t)' \Big( \Psi_{N,t}^0(0,u) - \Psi_t(0,u) \Big) C(\tfrac{u}{N}+t) \\
&\qquad\qquad + B(t)' \Psi_t(0,u) \Big( C(\tfrac{u}{N}+t) - C(t) \Big)\Big| \\				
&\leq  \mathbb{1}_{\{u\leq 0\}}\Big( \norm{B(t)} \norm{ \Psi_{N,t}^0(0,u) - \Psi_t(0,u)}   \Big(\sup_{s\in\R}\norm{C(s)}\Big)\\
&\qquad\qquad+ \norm{B(t)}\norm{ \Psi_t(0,u)}\norm{C(\tfrac{u}{N}+t) - C(t)}\Big) =:P_1+P_2.
\end{align*}
Since $C(\cdot)$ is continuous, we immediately obtain $P_2\rightarrow0$ as $N\rightarrow\infty$. In view of $P_1$ it is sufficient to show that for all $\varepsilon>0$ and sufficiently large $N$
\begin{align}\label{equation:inequalityPsilocstat}
\norm{\Psi_{N,t}^0(0,u)-\Psi_t(0,u)}\leq \varepsilon.
\end{align}
Due to the equivalence of all norms on $M_{p\times p}(\R)$, it is sufficient to show (\ref{equation:inequalityPsilocstat}) for the norm of each column. By $\Psi_{N,t}^{0(j)}(0,u)$ and $\Psi_t^{(j)}(0,u)$ we denote the $j$-th column, $j=1,\ldots,p$ of $\Psi_{N,t}^0(0,u)$ and $\Psi_t(0,u)$. Then, for functions $f_{N,t}(s,x)=A(\tfrac{s}{N}+t)x$ and $\tilde{f}_{t}(s,x)=A(t)x$ we obtain
\begin{align*}
\Psi_{N,t}^{0(j)}(u,u) &=e_j,  \quad &&\tfrac{d}{ds} \Psi_{N,t}^{0(j)}(s,u) = f_{N,t}\left(s,\Psi_{N,t}^{0(j)}(s,u)\right),\\
\Psi_t^{(j)}(u,u)&= e_j\quad \text{and} \quad && \tfrac{d}{ds} \Psi_t^{(j)}(s,u)= \tilde{f}_{t}\left(s,\Psi_t^{(j)}(s,u)\right),		
\end{align*}
where $e_j$ denotes the $j$-th unit vector. Note that $f_{N,t}$ and $\tilde{f}_{t}$ are Lipschitz continuous in the second argument with Lipschitz constant $L=\sup_{s\in[u,0]}\norm{A(\tfrac{s}{N}+t)}+A(t)<\infty$. Moreover,
\begin{align*}
\norm{ f_t\left(s,\Psi_{t}^{(j)}(s,u)\right) -  f_{N,t}\left(s,\Psi_{t}^{(j)}(s,u)\right)} \leq \delta \norm{\Psi_{t}^{(j)}(s,u)} \leq \delta c,\quad s\in[u,0],
\end{align*}
since $\norm{A(\tfrac{s}{N}+t)-A(t)}<\delta$ for any $\delta>0$ for sufficiently large $N$ and $\Psi_{t}^{(j)}(\cdot,u)$ is continuous and thus bounded on $[u,0]$. An application of \cite[§12.V.]{W2000} gives (\ref{equation:inequalityPsilocstat}).
Finally, by using the dominated convergence theorem with majorant
\begin{align*}
\abs{g_N^0(Nt,-u)} &\leq \mathbb{1}_{\{u\leq 0\}} \norm{B(t)}  \Big(\sup_{s\in\R}\norm{C(s)}\Big)  \norm{ \Psi_{N,t}^0(0,u)}  \\
		&\leq  \mathbb{1}_{\{u\leq 0\}} c_t  F_t(u)\in L^2(\R)
\end{align*}
for some constant $c_t>0$, where the last inequality follows from $(C1)$ and $(C2)$, we can deduce that $\norm{g_N^0(Nt,\cdot)-g(t,\cdot)}_{L^2}\rightarrow0$ as $N\rightarrow\infty$.
\end{proof}

\noindent In fact, assumption $(C3)$ in Proposition \ref{prop_C1_C2_C3_StateSpace_PeanoBaker} is an immediate consequence if the state space system is uniformly exponentially stable. 

\begin{Definition}[{\cite[Chapter 6, Definition 6.5 and Theorem 6.7)]{R1996}}]
A sequence of linear state space models as in Definition \ref{def_sequencetimevarying_StateSpace_PBS} is called uniformly exponentially stable, if there exist $\gamma>0$ and $\lambda>0$, such that
\begin{align*}
\norm{ \Psi_{N,t}^0(s,s_0)}  \leq \gamma e^{-\lambda(s-s_0)}
\end{align*}
for all $s,s_0$, where $s>s_0$, $N\in\N$ and $t\in\R$.
\end{Definition}

\begin{Corollary}\label{corollary:exponentialstableC3}
If a linear state space model as in Definition \ref{def_sequencetimevarying_StateSpace_PBS} is uniformly exponentially stable, then
\begin{align*}
		\norm{ \Psi_{N,t}^0(0,u)}  \leq \gamma e^{\lambda u} =: F_t(u)
		\in L^2((-\infty,0]), \text{ for all } N,t,u\leq0,
\end{align*}
which is $(C3)$ in Proposition \ref{prop_C1_C2_C3_StateSpace_PeanoBaker}.
\end{Corollary}

\begin{Proposition}\label{prop_C3_StateSpace_PeanoBaker}
Each of the following two conditions is sufficient for a state space model $\{Y_N(t), t\in\R\}_{N\in\N}$ as in Definition \ref{def_sequencetimevarying_StateSpace_PBS} to be uniformly exponentially stable.
\begin{enumerate}[label={(\alph*)}]
\item Let $\lambda_{max}(t)$, $t\in\R$ denote the largest eigenvalue of $A(t)+A(t)'$. If there exist positive constants $\gamma$ and $\lambda$, such that
\begin{align*}
\int_{s_0}^s \lambda_{max}(\tfrac{\nu}{N}+t) d\nu \leq -\lambda(s-s_0) + \gamma 
\end{align*}
for all $s, s_0,t$ and $N$ with $s\geq s_0$, then, due to \cite[Corollary 8.4]{R1996}, $Y_N(t)$ is uniformly exponentially stable.
\item Suppose A(t) is continuously differentiable and there exist positive constants $\alpha$, $\mu$, and $\beta$ such that $\norm{ A(t)} \leq\alpha$, $\norm{ \frac{d}{dt}A(t)} \leq\beta$ and the eigenvalues $\lambda_j(t)$ of $A(t)$ for $j=1,\ldots,p$ satisfy $\mathfrak{Re}(\lambda_j(t))\leq-\mu$ for all $t$. Then, due to \cite[Theorem 8.7]{R1996}, $Y_N(t)$ is uniformly exponentially stable.
\end{enumerate}
\end{Proposition}

\begin{Remark}
Part (b) of Proposition \ref{prop_C3_StateSpace_PeanoBaker} corresponds to condition $(C2)$ in Proposition \ref{prop_C1_C2_CAR1} for sequences of tvCAR(1) processes.
\end{Remark}

\begin{Remark}
There is a high structural resemblance of the above results to known results from the theory on locally stationary processes in discrete time. Indeed, conditions for time-varying AR(p) processes to be locally stationary as discussed in \cite[Theorem 2.3 ]{D1996} and \cite{K1995} are closely related to the conditions in Proposition \ref{prop_C3_StateSpace_PeanoBaker}.%on the eigenvalues of the companion matrix is stated and the coefficient functions of a time-varying AR(p) process are required to be continuous.
\end{Remark}

If the commutativity assumption (\ref{ass_commutative}) holds, the transition matrix is given by
\begin{align*}	
\Psi_{N,t}^0(0,u)=e^{\int_u^{0} A\left(\tfrac{s}{N}+t\right) ds}.
\end{align*}
Then, $Y_N(t)$ simplifies to
\begin{align*}
Y_N(t) = \int_{-\infty}^{Nt} B(t)' e^{\int_{-(Nt-u)}^{0} A\left(\tfrac{s}{N}+t\right) ds} C(\tfrac{u}{N}) L(du).
\end{align*}

%If the set $\{A(t)\}_{t\in\R}$ is mutually commutative, i.e. $A(t)A(s)=A(s)A(t)$ for all $t$, $s$, then it allows for the following corollary:

\begin{Proposition}\label{proposition:tvLDSParelocstatcommutative}
Let $\{Y_N(t), t\in\R\}_{N\in\N}$ be a sequence of time-varying state space processes as in Definition \ref{def_sequencetimevarying_StateSpace_PBS}. If $(C1)$ and $(C2)$ from Proposition \ref{prop_C1_C2_C3_StateSpace_PeanoBaker} hold, $\{A(t)\}_{t\in\R}$ is mutually commutative, the eigenvalues $\lambda_j(t)$ of $A(t)$ for $j=1,\ldots,p$ satisfy $\mathfrak{Re}(\lambda_j(t))\leq-\mu$ for all $t\in\R$ and some $\mu>0$ and either
\begin{enumerate}
\item[$(D1)$] $A(t)$ is diagonalizable for all $t\in\R$ or
%\item[$(D2)$] the family $\{A(t)\}$ is simultaneously Jordanizable in the sense that there exists a non-singular matrix $S$ such that $S^{-1}A(t)S$ is in Jordan normal form for all $t\in\R$ or
\item[$(D2)$] there exists $\tau>0$ such that $\sup_{\tau\leq x<\infty}\norm{\tfrac{1}{x}\int_\tau^xA\left(\tfrac{s}{N}+t\right)}<C$ for all $N$ and a constant $C>0$,
\end{enumerate}
then $Y_N(t)$ is locally stationary.
\end{Proposition}

\begin{proof}
It is sufficient to check $(C3)$ from Proposition \ref{prop_C1_C2_C3_StateSpace_PeanoBaker}. We start by assuming that $(D1)$ holds. Then, due to \cite[Theorem 1.3.12]{HJ1990}, $\{A(t)\}_{t\in\R}$ is simultaneously diagonalizable. Thus, there exists a non-singular matrix $S$ such that $S^{-1}A(\frac{s}{N}+t)S=diag(\lambda_1(\frac{s}{N}+t),\ldots,\lambda_p(\frac{s}{N}+t))=:D(\frac{s}{N}+t)$. Considering the spectral norm, we obtain for all $u\leq0$
\begin{align*}
\norm{ \Psi_{N,t}^0(0,u)}&=\norm{e^{\int_u^{0} A\left(\tfrac{s}{N}+t\right) ds}}=\norm{e^{\int_u^{0}S D\left(\tfrac{s}{N}+t\right) dsS^{-1}}}\leq\norm{S}\norm{S^{-1}} \norm{e^{\int_u^{0}D\left(\tfrac{s}{N}+t\right) ds}}\\
&\leq C \max \left\{ \sqrt{\mu}: \mu\in\sigma\left( e^{\int_u^0 D(\frac{s}{N}+t)^* + D(\frac{s}{N}+t) ds} \right) \right\} = C \max_{j=1,...,p} \sqrt{ e^{2\int_u^0 \mathfrak{Re}(\lambda_j(\frac{s}{N}+t)) ds} }\\
&= C \max_{j=1,...,p} e^{\int_u^0 \mathfrak{Re}(\lambda_j(\frac{s}{N}+t)) ds} \leq C e^{\int_u^0 -\mu ds} = C e^{\mu u}
\end{align*}
for some constant $C>0$.\\
In the case where $(D2)$ holds, we have
\begin{align*}
\norm{ \Psi_{N,t}^0(0,u)}&=\norm{e^{\int_u^{0} A\left(\frac{s}{N}+t\right) ds}}=\norm{e^{\int_0^{-u} A\left(\frac{-s}{N}+t\right) ds}}\\
&\leq \mathbb{1}_{\left\{-u\in[0,\tau]\right\}}e^{\abs{\tau}\sup_{s\in[0,\tau]}\norm{A\left(\frac{-s}{N}+t\right)}}+ \mathbb{1}_{\{-u> \tau\}}e^{\abs{\tau}\sup_{s\in[0,\tau]}\norm{A\left(\frac{-s}{N}+t\right)}} \norm{e^{\int_\tau^{-u}A\left(\frac{-s}{N}+t\right)ds}},
\end{align*}
where we used that the integrals $\int_0^{\tau}A\left(\frac{-s}{N}+t\right)ds$ and $\int_\tau^{-u}A\left(\frac{-s}{N}+t\right)ds$ commute. Therefore, it is sufficient to bound $\norm{e^{\int_\tau^{-u}A\left(\frac{-s}{N}+t\right)ds}}$. In the following we use \cite[Theorem 7.7.1]{L1982}. Since the family $\{A(t)\}_{t\in\R}$ is mutually commutative, the family can be reduced simultaneously to an upper triangular form by a single unitary transformation, i.e. there exists a unitary matrix $U\in M_{p\times p}(\C)$ such that $U^*A(t)U=T(t)$ is an upper triangular matrix for all $t\in\R$ (see \cite[Theorem 2.3.3]{HJ1990}). %Moreover, $T(t)$ is of the form $diag(\lambda_1(t),\ldots,\lambda_p(t))+N(t)$, where $N(t)$ is some nilpotent matrix. 
For each $x>\tau$, the diagonal entries (and hence also the eigenvalues) of
$\tfrac{1}{x}\int_\tau^xT\left( \tfrac{s}{N}+t \right)ds$ are $\tfrac{1}{x}\int_\tau^x \lambda_1\left(\tfrac{s}{N}+t\right)ds,\ldots,\tfrac{1}{x}\int_\tau^x \lambda_p\left(\tfrac{s}{N}+t\right)ds$. These are also the eigenvalues of $\tfrac{1}{x}\int_\tau^x A\left( \tfrac{s}{N}+t \right)ds$ since
\begin{align*}
\tfrac{1}{x}\int_\tau^xT\left( \tfrac{s}{N}+t \right)ds=\tfrac{1}{x}U^*\int_\tau^x A\left( \tfrac{s}{N}+t \right)ds\ U.
\end{align*}
For the real part of the eigenvalues we obtain
\begin{align*}
\mathfrak{Re}\left(\tfrac{1}{x}\int_\tau^x \lambda_i\left(\tfrac{s}{N}+t\right)ds\right)=\tfrac{1}{x}\int_\tau^x \mathfrak{Re}\left(\lambda_i\left(\tfrac{s}{N}+t\right)\right)ds\leq-\mu \left(\frac{x-\tau}{x}\right)\leq-\mu
\end{align*}
for all $i=1,\ldots,p$, $N\in\N$ and $x\in\R$. Hence
\begin{align*}
\overline{\bigcup_{\tau\leq t<\infty}\tilde\sigma\left(\frac{1}{x}\int_\tau^x A\left( \tfrac{s}{N}+t \right)ds\right)}\subset\overline{\bigcup_{\tau\leq t<\infty}\sigma\left(\frac{1}{x}\int_\tau^x A\left( \tfrac{s}{N}+t \right)ds\right)}\subset \left\{z\in\C:\mathfrak{Re}(z)\leq-\mu\right\}
\end{align*}
for all $N\in\N$, where $\tilde\sigma(B)$ denotes the collection of all distinct eigenvalues of the matrix $B$ and $\sigma(B)$ the spectrum $B$. Finally, an application of \cite[Theorem 7.7.1]{L1982} gives
\begin{align*}
\norm{e^{\int_\tau^{-u}A\left(\frac{-s}{N}+t\right)ds}}\leq Ce^{\mu u}
\end{align*}
for some constant $C>0$.
\end{proof}

Sequences of time-varying CARMA processes, i.e. where the family $\{\mathcal{A}(t)\}_{t\in\R}$ forms a family of companion matrix, cannot be covered by Proposition \ref{proposition:tvLDSParelocstatcommutative}, since companion matrices are in general not commutative. The following Proposition brings further insight when a family of companion matrices is mutually commutative.

\begin{Proposition}\label{Proposition:commutativitycompanion}
Let $\{\mathcal{A}(t)\}_{t\in\R}$ be a family of companion matrices and $\tau\in\R$ fixed. For any $t\in\R$ the matrix $\mathcal{A}(t)$ commutes with $\mathcal{A}(\tau)$ if and only if it is a polynomial of $\mathcal{A}(\tau)$ over $\C$.
\end{Proposition}
\begin{proof}
It is clear that any polynomial of $\mathcal{A}(\tau)$ commutes with $\mathcal{A}(\tau)$. For the other direction we refer to \cite[Exercise 3.3P17]{HJ1990}.
\end{proof}

If, for a sequence of time-varying CARMA processes, the family $\{\mathcal{A}(t)\}_{t\in\R}$ is not mutually commutative, Proposition \ref{prop_C1_C2_C3_StateSpace_PeanoBaker} provides sufficient conditions for local stationarity, where condition $(C3)$ can be derived from Proposition \ref{prop_C3_StateSpace_PeanoBaker} and Corollary \ref{corollary:exponentialstableC3}.
%
%
%Now, we turn again to sequences of time-varying CARMA processes, 
%\textcolor[rgb]{1,0,0}{which can not be depicted by the above corollary since companion matrices are in general not commutative.}
%\begin{Corollary}
%A sequence of time-varying CARMA($p,q$) processes ($p>q$) as in Definition \ref{def_sequencetimevarying_StateSpace_PBS} with coefficient functions $(\mathcal{B}(t),\mathcal{A}(t),\mathcal{C})$ as in (\ref{eq_statespace_tilde}) is locally stationary according to Definition~\ref{def_locstat_cont}, if
%\begin{enumerate}
%\item[$(C1)$] $\mathcal{A}(\cdot)$ and $\mathcal{B}(\cdot)$ are continuous
%\item[$(C2)$] $\norm{\mathcal{B}(s)}<\infty$ for all $s\in\R$
%\item[$(C3)$] $\norm{ \Psi_{N,t}^0(0,u)}  \leq F_t(u)$ for some function $F_t\in L^2((-\infty,0],\R)$ for each $N\in\N$ and $t\in\R$.
%\end{enumerate}
%\end{Corollary}
%
%According to Remark \ref{prop_C3_StateSpace_PeanoBaker} (ii) condition $(C3)$ can be achieved by demanding among other things that all pointwise eigenvalues of $A(t)$ satisfy $\mathfrak{Re}(\lambda(t))\leq -\mu$ for some positive constant $\mu$. This resembles condition $(C2)$ in Proposition~\ref{prop_C1_C2_CAR1} for sequences of time-varying CAR(1) processes.
%

\section{Time-varying spectrum}\label{sec5}
For a stationary processes $\{Y(t),t\in\R\}$ the autocovariance function $\gamma_Y(h):=Cov(Y(t+h),Y(t))$ is related to the spectral density $f_Y(\lambda)$ by
\begin{align*}	
\gamma_Y(h)=\int_{-\infty}^\infty e^{ih\lambda} f_Y(\lambda) d\lambda \qquad \text{and} \qquad f_Y(\lambda)	=\frac{1}{2\pi} \int_{-\infty}^\infty e^{-ih\lambda} \gamma_Y(h) dh.
\end{align*}
%see \cite{priestley_spectral_1994} Section~4.8. \\
To describe the time-varying spectrum of a discrete-time locally stationary time series, \cite{D1996} used the Wigner-Ville spectrum (see also \cite{BT2004, FM1984, FM1985}). A comparable approach was presented in \cite[Section 11.2]{P1994}, where the author used the evolutionary spectrum. However, in contrast to this approach, the Wigner-Ville spectrum has the important consequence of a unique spectral representation as discussed in \cite[p. 74]{BT2004} and \cite[p. 143]{D1996}. In view of this property, we follow the approach of \cite{D1996} and define the Wigner-Ville spectrum and time-varying spectral density for a continuous-time locally stationary process as follows.

\begin{Definition}
Let $\{Y_N(t), t\in\R\}_{N\in\N}$ be a sequence of locally stationary processes. For $N\in\N$ we define the Wigner-Ville spectrum as
\begin{align*}
f_N(t,\lambda)=\frac{1}{2\pi} \int_{-\infty}^\infty e^{-i\lambda s} Cov\left(Y_N(t+\tfrac{s}{2N}),Y_N(t-\tfrac{s}{2N})\right) ds,
\end{align*}
and the (time-varying) spectral density of the process $Y_N(t)$ as
\begin{align*}
f(t,\lambda)= \frac{\Sigma_L}{2\pi} |A(t,\lambda)|^2,
\end{align*}
where $A(t,\lambda)$ denotes the limiting transfer function from Definition \ref{def_locstat_cont2}.
\end{Definition}

The following theorem is a continuous-time analogue to \cite[Theorem 2.2]{D1996}.

\begin{Theorem} \label{thm_WignerVille}
Let  $\{Y_N(t), t\in\R\}_{N\in\N}$ be a sequence of locally stationary processes in the form (\ref{equation:locallystationaryprocessfrequencydomain}). If
\begin{enumerate}[label={(\alph*)}]
\item $\norm{A_N^0(N(t\pm\tfrac{s}{2N}),\cdot)-A(t,\cdot)}_{L^2}\underset{N\rightarrow\infty}{\longrightarrow}0$ for all $s,t\in\R$,
\item $A_N^0$ as well as $A$ are uniformly bounded in $L^2$, i.e $\norm{A_N^0(Nt,\cdot)}_{L^2}\leq K$, $\norm{A(t,\cdot)}_{L^2}\leq K$ for all $t\in\R$, $N\in\N$ and a constant $K>0$ and
\item $A_N^0(Nt,\cdot)$ as well as $A(t,\cdot)$ are differentiable for all $t\in\R$, $N\in\N$ and the derivatives $\tfrac{d}{d\mu}A_N^0(Nt,\mu)$, $\tfrac{d}{d\mu}A(t,\mu)$ are uniformly bounded in $L^2$, i.e. $\norm{\tfrac{d}{d\mu}A_N^0(Nt,\cdot)}_{L^2}\leq K$ and $\norm{\tfrac{d}{d\mu}A(t,\cdot)}_{L^2}\leq K$, for all $t,N$ and a constant $K>0$,
\end{enumerate}
then the Wigner-Ville spectrum tends pointwise for each $t\in\R$ in mean square to the time-varying spectral density, i.e.
\begin{align*}
\int_\R \abs{f_N(t,\lambda)-f(t,\lambda)}^2 d\lambda \overset{N\rightarrow\infty}{\longrightarrow} 0.
\end{align*}
\end{Theorem}

\begin{Remark}
Since $A_N^0$ and $A$ are defined as Fourier transforms of $g_N^0$ and $g$ in $L^2$, they exist as elements in $L^2$, i.e. as representatives of equivalence classes. As usual, this does not allow for taking derivatives in the usual sense, but would lead to the concept of weak derivatives. \\
However, for a function $f\in L^1$ such that $uf(u)\in L^1$, the derivative of the Fourier transform $\widehat{f}(\mu)$ can be expressed as $\tfrac{d}{d\mu}\widehat{f}(\mu) = \widehat{(-iuf(u))}(\mu)$, due to \cite[Theorem 1.6, Chapter VI]{K2004}. An application of this theorem to the Fourier transform pairs $A_N^0$ and $g_N^0$ as well as $A$ and $g$, ensures the existence of the pointwise derivatives in (c). The conditions on the the kernel functions can be readily obtained for instance if the considered sequence of locally stationary state space models is uniformly exponentially stable, since then the kernel functions are of exponential decay (see Corollary \ref{cor_WignerVille_i_ii_iii_StateSpace}).
\end{Remark}
%\begin{Remark} To check the assumptions (a)-(c), it is possible to show pointwise convergence of $A_N^0$ to $A$ as $N\rightarrow\infty$ and to find a square integrable majorant $|A_N^0(Nt,\mu)|\leq F(\mu)$ for all $t\in\R$ and $N\in\N$ (or equivalently for the kernel functions $g_N^0$ and $g$). Then (a)-(c) follow immediately from the dominated convergence theorem (see Corollary \ref{cor_WignerVille_i_ii_iii_StateSpace}).
%\end{Remark}

\begin{Lemma}\label{lem_1}
A continuous function $h:\R\rightarrow\R$ in $L^2(\R)$ lies also in $L^p(\R)$ for all $p>2$.
\end{Lemma}

\begin{proof}
Define the set $J=\{x, \abs{f(x)}\geq1 \}$. Since $\lim_{x\rightarrow\pm\infty}f(x)=0$, $J$ is a compact set. Thus, there exists a constant $C$ such that $\abs{f(x)}\leq C$ for all $x\in J$. Finally,
\begin{align*}
\int_\R \abs{f(x)}^p dx \leq \int_J C^p dx + \int_{\R\backslash J} \abs{f(x)}^2 dx < \infty.
\end{align*}
\end{proof}

\noindent\textit{Proof of Theorem \ref{thm_WignerVille}.}
In the following, $K$ denotes the constant from the conditions (b) and (c). First, we note that the covariance of $Y_N(t)$ is given by
\begin{align*}	
Cov\left(Y_N(t_1),Y_N(t_2)\right) = \frac{\Sigma_L}{2\pi} \int_\R e^{i\mu (t_1-t_2) N} A_N^0(Nt_1,\mu) \overline{A_N^0(Nt_2,\mu)} d\mu.
\end{align*}
Then, we obtain for the Wigner-Ville spectrum and the (time-varying) spectral density
\begin{align*}
f_N(t,\lambda)&= \frac{1}{2\pi} \int_{\R} e^{-i\lambda s} \frac{\Sigma_L}{2\pi} \int_{\R} e^{i\mu \left((t+\tfrac{s}{2N})-(t-\tfrac{s}{2N})\right) N} A_N^0(N(t+\tfrac{s}{2N}),\mu) \overline{A_N^0(N(t-\tfrac{s}{2N}),\mu)} d\mu\ ds \\
&= \frac{\Sigma_L}{(2\pi)^2} \int_{\R} e^{-i\lambda s} \int_{\R} e^{i\mu s} A_N^0(N(t+\tfrac{s}{2N}),\mu) \overline{A_N^0(N(t-\tfrac{s}{2N}),\mu)} d\mu\ ds\text{ and}\\
f(t,\lambda)&= \frac{\Sigma_L}{2\pi} |A(t,\lambda)|^2
					%= \Sigma_L \left( \frac{1}{\sqrt{2\pi}} \int_{s\in\R} e^{-i\lambda s} \widehat{\varphi}(s) ds \right) \\
					%&= \frac{\Sigma_L }{\sqrt{2\pi}} \int_{s\in\R} e^{-i\lambda s} \left( \frac{1}{\sqrt{2\pi}} \int_{\mu\in\R} e^{i\mu s} \varphi(\mu) d\mu \right) ds \\
= \frac{\Sigma_L }{(2\pi)^2} \int_{\R} e^{-i\lambda s} \int_{\R} e^{i\mu s} |A(t,\lambda)|^2 d\mu\ ds.
\end{align*}
We note that, due to the differentiability condition (c), the function $A(t,\cdot)$ is continuous and in $L^2$. Then, an application of Lemma \ref{lem_1} gives $A(t,\cdot)\in L^4$, which implies $f(t,\cdot)\in L^2$. Moreover, for all $t\in\R$, $N\in\N$ we obtain from Plancherel's theorem, the Cauchy–Schwarz inequality and the integration by parts formula for some $C>0$
\begin{align*}
&\int_{\R} \abs{f_N(t,\lambda)}^2 d\lambda= \frac{\Sigma_L}{4\pi^2} \int_{\R} \abs{ \int_{\R} e^{-i\lambda s} \int_{\R} e^{i\mu s}  A_N^0(N(t+\tfrac{s}{2N}),\mu) \overline{A_N^0(N(t-\tfrac{s}{2N}),\mu)} d\mu\ ds}^2 d\lambda \\
&=C \int_{\R} \abs{ \int_{\R} e^{i\mu s}  A_N^0(N(t+\tfrac{s}{2N}),\mu) \overline{A_N^0(N(t-\tfrac{s}{2N}),\mu)} d\mu }^2 ds \\
&\leq C \int_{|s|< 1} \|A_N^0(N(t+\tfrac{s}{2N}),\cdot)\|_{L^2}^2 \|A_N^0(N(t-\tfrac{s}{2N}),\cdot)\|_{L^2}^2 ds\\
&\quad+ C \int_{|s|\geq 1} \abs{ \int_{\R} e^{i\mu s} A_N^0(N(t+\tfrac{s}{2N}),\mu) \overline{A_N^0(N(t-\tfrac{s}{2N}),\mu)} d\mu  }^2 ds \\
&\leq 2C K^4+ C \int_{|s|< 1} \Bigg|0-\int_{\R} \frac{e^{i\mu s}}{is} \bigg(A_N^0(N(t+\tfrac{s}{2N}),\mu)  \frac{d}{d\mu}\overline{A_N^0(N(t+\tfrac{s}{2N}),\mu)} \\
&\qquad\qquad\qquad\qquad\qquad\qquad+\frac{d}{d\mu} A_N^0(N(t+\tfrac{s}{2N}),\mu)  \overline{A_N^0(N(t+\tfrac{s}{2N}),\mu)}  \bigg)d\mu\Bigg|^2 ds\\
&\leq 2C K^4
+ 4CK^4 \int_{|s|\geq 1} \frac{1}{s^2}ds<\infty,
\end{align*}
where $\lim_{\mu\rightarrow\pm\infty}A_N^0(t,\mu)=0$, since $A_N^0(t,\cdot)$ is continuous and in $L^2(\R)$ for all $t\in\R$.
Finally, the integral $\int_\R |f_N(t,\lambda)-f(t,\lambda)|^2 d\lambda$ is well defined, since $f_N(t,\cdot)$ and $f(t,\cdot)$ are both in $L^2(\R)$ for all $t\in\R$.
From Plancherel's theorem we obtain
\begin{align*}
&(2\pi)^2 \int_\R \abs{f_N(t,\lambda)-f(t,\lambda)}^2 d\lambda \\
&= \frac{\Sigma_L^2}{2\pi} 
\int_{\R} \abs{ \frac{1}{\sqrt{2\pi}}\int_{\R} e^{-i\lambda s} \left( \int_{\R} e^{i\mu s} \left( A_N^0(N(t+\tfrac{s}{2N}),\mu) \overline{A_N^0(N(t-\tfrac{s}{2N}),\mu)} - A(t,\mu) \overline{A(t,\mu)} \right) d\mu \right) ds }^2 d\lambda \\
&= \frac{\Sigma_L^2}{2\pi} \int_{\R} 
		\abs{ \int_{\R} e^{i\mu s} \left( A_N^0(N(t+\tfrac{s}{2N}),\mu) \overline{A_N^0(N(t-\tfrac{s}{2N}),\mu)} - A(t,\mu) \overline{A(t,\mu)} \right) d\mu }^2 ds\\
		& =: \frac{\Sigma_L^2}{2\pi} \int_{\R} \abs{ \widehat{a}_N(\tfrac{s}{2N})}^2 ds,
\end{align*}
where $\widehat{a}_N(\tfrac{s}{2N}) 	= \int_{\R} e^{i\mu s} a_N(\tfrac{s}{2N},\mu) d\mu$. It is left to show 
\begin{align}\label{equation:ahatconverges0}
\int_{\R} \abs{ \widehat{a}_N(\tfrac{s}{2N})}^2 ds\underset{N\rightarrow\infty}{\longrightarrow}0.
\end{align} 
The proof of (\ref{equation:ahatconverges0}) consists of several steps. We start by showing $\widehat{a}_N(\tfrac{s}{2N}) \underset{N\rightarrow\infty}{\longrightarrow}0$. Indeed, for fixed $s,t\in\R$ we obtain
\begin{align*}
\abs{\widehat{a}_N(\tfrac{s}{2N})} &\leq \int_{\R} \abs{A_N^0(N(t+\tfrac{s}{2N}),\mu) \overline{A_N^0(N(t-\tfrac{s}{2N}),\mu)} - A(t,\mu) \overline{A(t,\mu)}} d\mu \\
&\leq \int_{\R} \abs{A_N^0(N(t+\tfrac{s}{2N}),\mu) - A(t,\mu)}\abs{\overline{A_N^0(N(t-\tfrac{s}{2N}),\mu)}} \\
&\qquad+  \abs{A(t,\mu)}\abs{\overline{A_N^0(N(t-\tfrac{s}{2N}),\mu)} - \overline{A(t,\mu)}} d\mu \\
&\leq \norm{ A_N^0(N(t+\tfrac{s}{2N}),\cdot) - A(t,\cdot)}_{L^2} \norm{ A_N^0(N(t-\tfrac{s}{2N}),\cdot)}_{L^2} \\
&\qquad+ \norm{ A(t,\cdot)}_{L^2}\norm{ A_N^0(N(t-\tfrac{s}{2N}),\cdot) - A(t,\cdot)}_{L^2}.
\end{align*}
Now, due to condition (a), it holds $\norm{ A_N^0(N(t-\tfrac{s}{2N}),\cdot) - A(t,\cdot)}_{L^2}\rightarrow0$ as $N\rightarrow\infty$. Moreover, using the conditions (b) and (c), there exists a constant $D$, which may depend on $s$ and $t$, such that 
\begin{align*}
\norm{A(t,\cdot)}_{L^2} &\leq D \text{ and}\\
 \norm{A_N^0(N(t\pm\tfrac{s}{2N}),\cdot)}_{L^2} &\leq \norm{A_N^0(N(t\pm\tfrac{s}{2N}),\cdot)-A(t,\cdot)}_{L^2} + \norm{A(t,\cdot)}_{L^2} \leq D
\end{align*}
for sufficiently large $N$. Thus
\begin{align}
\begin{aligned}\label{equation:ahatconvergespointwiseto0}
\abs{\widehat{a}_N(\tfrac{s}{2N})}&\leq\norm{A_N^0(N(t+\tfrac{s}{2N}),\cdot)-A(t,\cdot)}_{L^2}^2 \norm{A_N^0(N(t-\tfrac{s}{2N}),\cdot)}_{L^2}^2\\
&\quad + \norm{ A(t,\cdot)}_{L^2}^2 \norm{A_N^0(N(t-\tfrac{s}{2N}),\cdot)-A(t,\cdot)}_{L^2}^2\rightarrow 0,
\end{aligned}
\end{align}
as $N\rightarrow\infty$.\\
Next, we show $|\widehat{a}_N(\tfrac{s}{2N})|\leq\frac{E}{|s|}$, for all $s\in\R$, sufficiently large $N\in\N$ and some constant $E>0$, which may depend on $t$. On the one hand, we have
\begin{align}
\begin{aligned}\label{equation:ahatbounded}
\int_{\R} \abs{\tfrac{d}{d\mu} a_N(\tfrac{s}{2N},\mu)} d\mu&= \int_{\R} \bigg|\left(\tfrac{d}{d\mu} A_N^0(N(t+\tfrac{s}{2N}),\mu)\right) \overline{A_N^0(N(t-\tfrac{s}{2N}),\mu)} \\
&\qquad+ A_N^0(N(t+\tfrac{s}{2N}),\mu) \left(\tfrac{d}{d\mu} \overline{A_N^0(N(t-\tfrac{s}{2N}),\mu)}\right) \\
&\qquad- \left(\tfrac{d}{d\mu}A(t,\mu)\right) \overline{A(t,\mu)} - A(t,\mu) \left(\tfrac{d}{d\mu} \overline{A(t,\mu)}\right)\bigg| d\mu \\
&\leq \norm{ \tfrac{d}{d\mu} A_N^0(N(t+\tfrac{s}{2N}),\cdot)}_{L^2} 
\norm{ A_N^0(N(t-\tfrac{s}{2N}),\cdot) }_{L^2} \\
&\qquad+ \norm{A_N^0(N(t+\tfrac{s}{2N}),\cdot) }_{L^2}\norm{\tfrac{d}{d\mu} A_N^0(N(t-\tfrac{s}{2N}),\cdot)}_{L^2} \\
&\qquad+ 2 \norm{ A(t,\cdot)}_{L^2}\norm{\tfrac{d}{d\mu} A(t,\cdot)}_{L^2}\leq E,
\end{aligned}
\end{align}
where the last inequality follows from (b) and (c). On the other hand, $\tfrac{d}{d\mu} a_N(\tfrac{s}{2N},\mu)\in L^1(\R)$, since $A_N^0(t,\cdot),A(t,\cdot),\tfrac{d}{d\mu}A_N^0(t,\cdot),\tfrac{d}{d\mu}A(t,\cdot)\in L^2(\R)$, such that
\begin{align}
\begin{aligned}\label{equation:integrationbypartsahat}
\int_{\R} e^{i\mu s} \left( \tfrac{d}{d\mu} a_N(\tfrac{s}{2N},\mu) \right) d\mu
&=\left[ e^{i\mu s} a_N(\tfrac{s}{2N},\mu) \right]\bigg|_{\mu=-\infty}^\infty - \int_{\R} (is) e^{i\mu s} a_N(\tfrac{s}{2N},\mu) d\mu \\
&= (-is) \widehat{a}_N(\tfrac{s}{2N}),
\end{aligned}
\end{align}
where the limit in the first term of the partial integration is zero, because $a_N(\tfrac{s}{2N},\mu)$ is continuous and in $L^1(\R)$. Combining (\ref{equation:ahatbounded}) and (\ref{equation:integrationbypartsahat}) we obtain
\begin{align*}
\abs{\widehat{a}_N(\tfrac{s}{2N})}\leq \frac{1}{\abs{s}} \int_{\R} \abs{ \tfrac{d}{d\mu} a_N(\tfrac{s}{2N},\mu)} d\mu\leq \frac{E}{\abs{s}}.
\end{align*}
Finally, for $s^*\in\R$ 
\begin{align*}
(2\pi)^2 \int_\R \abs{f_N(t,\lambda)-f(t,\lambda)}^2 d\lambda &= \frac{\Sigma_L^2}{2\pi} \int_{s\in\R} \abs{ \widehat{a}_N(\tfrac{s}{2N})}^2 ds \\
&= \frac{\Sigma_L^2}{2\pi} \int_{\abs{s}\geq s^\ast} \abs{ \widehat{a}_N(\tfrac{s}{2N})}^2 ds+\frac{\Sigma_L^2}{2\pi} \int_{\abs{s}<s^\ast} \abs{ \widehat{a}_N(\tfrac{s}{2N})}^2 ds\\
&\leq \frac{\Sigma_L^2E^2}{\pi s^*}+\frac{\Sigma_L^2}{2\pi} \int_{\abs{s}<s^\ast} \abs{ \widehat{a}_N(\tfrac{s}{2N})}^2 ds.
\end{align*}
The second term converges to zero by the dominated convergence theorem, where pointwise convergence follows from (\ref{equation:ahatconvergespointwiseto0}) and a convergent majorant can be obtained from the boundedness conditions in (b) and the Cauchy-Schwarz inequality, noting that the support of the integral is compact. Therefore, for all $\varepsilon>0$ and sufficiently large $s^*$ and $N$ it holds
\begin{align*}
 \int_\R \abs{f_N(t,\lambda)-f(t,\lambda)}^2 d\lambda &\leq \frac{\Sigma_L^2E^2}{\pi s^*}+\frac{\Sigma_L^2}{2\pi} \int_{\abs{s}<s^\ast} \abs{ \widehat{a}_N(\tfrac{s}{2N})}^2 ds \leq \varepsilon,
\end{align*}
which concludes the proof.\qed

\begin{Corollary} \label{cor_WignerVille_i_ii_iii_StateSpace}
Let $Y_N(t)$ be a sequence of time-varying linear state space processes as in Definition \ref{def_sequencetimevarying_StateSpace_PBS}, such that both IVPs in (\ref{equ_IVP_PBS}) are uniformly exponentially stable, the conditions (C1)-(C3) from Proposition \ref{prop_C1_C2_C3_StateSpace_PeanoBaker} hold and $\sup_{t\in\R}\norm{B(t)}<\infty$. Then, the sequence of Wigner-Ville spectra tends in mean square to the time-varying spectral density.
\end{Corollary}

\begin{proof}
It is sufficient to check the conditions (a), (b) and (c) from Theorem \ref{thm_WignerVille}.
\begin{enumerate}[label={(\alph*)}]
\item For $s,t\in\R$ we obtain from Plancherel's theorem
\begin{align*}
&\norm{A_N^0(N(t\pm\tfrac{s}{2N}),\cdot)-A(t,\cdot)}_{L^2}^2 = 4\pi^2 \norm{g_N^0(N(t\pm\tfrac{s}{2N}),\cdot)-g(t,\cdot)}_{L^2}^2 \\
		&= 4\pi^2 \int_\R \mathbb{1}_{\{u\leq 0\}} \Big| B(t\pm\tfrac{s}{2N})' \Psi_{N,t\pm\tfrac{s}{2N}}^0(0,u) C(\tfrac{u}{N}+t\pm\tfrac{s}{2N})- B(t)' \Psi_{t}(0,u) C(t) \Big|^2 du,
\end{align*}
which tends to zero as $N\rightarrow\infty$ by the dominated convergence theorem. Pointwise convergence is secured by the continuity of $A,B$ and $C$ in (C1) and the continuity of the solution of an IVP on the input (see the proof of Proposition \ref{prop_C1_C2_C3_StateSpace_PeanoBaker}). Since the sequence $Y_N(t)$ is uniformly exponentially stable, we have $\norm{ \Psi_{N,t}^0(s_1,s_0)} \leq \gamma e^{-\lambda(s_1-s_0)}$ for some $\gamma, \lambda >0$ and all $s_1>s_0$. Therefore, a convergent majorant can be obtained by noting that
\begin{align*}
\mathbb{1}_{\{u\leq 0\}} \abs{ B(t\pm\tfrac{s}{2N})' \Psi_{t\pm\tfrac{s}{2N}}(0,u) C(\tfrac{u}{N}+t\pm\tfrac{s}{2N})}\leq \mathbb{1}_{\{u\leq 0\}} \left(\sup_{t\in\R}\|B(t)\|\right) \gamma e^{\lambda u} \left(\sup_{t\in\R}\norm{C(t)}\right).
\end{align*}
\item For $t\in\R$ and $N\in\N$ it holds $\norm{ \Psi_{N,t}^0(s_1,s_0)} \leq \gamma e^{-\lambda(s_1-s_0)}$ and $\norm{ \Psi_t(s_1,s_0)} \leq \gamma e^{-\lambda(s_1-s_0)}$ for some $\gamma, \lambda >0$ and all $s_1>s_0$. Thus
\begin{align*}
\norm{A_N^0(Nt,\cdot)}_{L^2}^2 &= 4\pi^2 \norm{g_N^0(Nt,\cdot)}_{L^2}^2 = 4\pi^2 \int_\R \mathbb{1}_{\{u\leq 0\}} \abs{ B(t)' \Psi_{N,t}^0(0,u) C(\tfrac{u}{N}+t)}^2 du \\
&\leq% 2\pi \int_{-\infty}^0  \left(\sup_{t\in\R}\norm{B(t)}\right)^2 \gamma^2 e^{2\lambda u} \left(\sup_{t\in\R}\norm{C(t)}\right)^2 du \\
%&=
\frac{2\pi^2\gamma^2}{\lambda} \left(\sup_{t\in\R}\|B(t)\|\right)^2 \left(\sup_{t\in\R}\|C(t)\|\right)^2 <\infty \text{ and}\\
\norm{A(t,\cdot)}_{L^2}^2&= 4\pi^2 \norm{g(t,\cdot)}_{L^2}^2 = 4\pi^2 \int_\R \mathbb{1}_{\{u\leq 0\}} \norm{ B(t)' \Psi_t(0,u) C(u)}^2 du \\
&\leq %2\pi \int_{-\infty}^0  \left(\sup_{t\in\R}\norm{B(t)}\right)^2 \gamma^2 e^{2\lambda u} \left(\sup_{t\in\R}\norm{C(t)}\right)^2 du \\
%&= 
\frac{2\pi^2\gamma^2}{\lambda} \left(\sup_{t\in\R}\norm{B(t)}\right)^2\left(\sup_{t\in\R}\norm{C(t)}\right)^2 <\infty.
\end{align*}
\item Since $Y_N(t)$ is uniformly exponentially stable, \cite[Theorem 1.6]{K2004} implies that %the Fourier transform $\widehat{f}$ of a function is differentiable with derivative $\tfrac{d}{d\xi}\widehat{f}(\xi)=\widehat{(-iuf(u))}(\xi)$. Here, this is ensured for $g_N^0$ and $g$ by the uniform exponential stability and thus
\begin{align*}
\tfrac{d}{d\mu} A_N^0(Nt,\mu) &= \int_\R e^{-i\mu u} (-iu) g_N^0(t,u) du \text{ and} \\
\tfrac{d}{d\mu} A(t,\mu) &= \int_\R e^{-i\mu u} (-iu) g(t,u) du,
\end{align*}
which are again in $L^2(\R)$, since
\begin{align*}
\norm{ \tfrac{d}{d\mu} A_N^0(Nt,\cdot)}_{L^2}^2&= 4\pi^2 \norm{(-i\cdot) g_N^0(Nt,\cdot)}_{L^2}^2 \\
&= 4\pi^2 \int_\R \mathbb{1}_{\{u\leq 0\}} \abs{ (-iu) B(t)' \Psi_{N,t}^0(0,u) C(\tfrac{u}{N}+t) }^2 du \\
&\leq 4\pi^2 \Big(\sup_{t\in\R}\norm{B(t)}\Big)^2  \left(\gamma^2 \int_{-\infty}^0 u^2 e^{2\lambda u} du\right) \Big(\sup_{t\in\R}\norm{C(t)}\Big)^2 \\
&= \frac{\gamma^2\pi^2}{\lambda^3} \Big(\sup_{t\in\R}\norm{B(t)}\Big)^2\Big(\sup_{t\in\R}\norm{C(t)}\Big)^2<\infty\text{ and analogously} \\
\norm{\tfrac{d}{d\mu} A(t,\cdot)}_{L^2}^2&\leq \frac{\gamma^2\pi^2}{\lambda^3} \Big(\sup_{t\in\R}\norm{B(t)}\Big)^2\Big(\sup_{t\in\R}\norm{C(t)}\Big)^2<\infty.
\end{align*}
%Furthermore, $\tfrac{d}{d\mu} A(t,\mu)$ is continuous in $t$ and
%\[	\|\tfrac{d}{d\mu} A_N^0(Nt,\cdot)-\tfrac{d}{d\mu} A(t,\cdot)\|_{L^2(\R,\C)} \xrightarrow{N\rightarrow\infty}0 \]
%by dominated convergence (analogously to part (i)). 
\end{enumerate}
\end{proof}

\section*{Acknowledgements}
The third author was supported by the scholarship program of the Hanns-Seidel Foundation, funded by the Federal Ministry of Education and Research.

\bibliographystyle{acm}

\end{document}